\documentclass[12pt]{amsart}
\usepackage{graphicx}
\usepackage{amscd}
\usepackage[top=1in,bottom=1in,left=1.15in,right=1.15in]{geometry}
\usepackage{amsmath,amssymb,amsfonts}
\usepackage{amsthm}
\usepackage{color}
\usepackage{hyperref}
\usepackage{tikz-cd}
\usetikzlibrary{plotmarks}
\usepackage{epstopdf}
\usetikzlibrary{matrix}
\usepackage{verbatim}

\newtheorem{theo}{Theorem}[section]
\newtheorem{prop}[theo]{Proposition}
\newtheorem{lemma}[theo]{Lemma}
\newtheorem{defn}[theo]{Definition}
\newtheorem{rem}[theo]{Remark}

\allowdisplaybreaks

\def\Ptwist{\unitlength.1em
  \begin{minipage}{15\unitlength}
    \begin{picture}(15,15)
     \put(0,7.5){\line(1,0){6}}
     \put(9,7.5){\line(1,0){6}}
      \put(4,9.5){\line(1,0){3}}
      \qbezier(7,9.5)(7,9.5)(9,5.5)
      \put(9, 5.5){\line(1,0){3}}
    \end{picture}
  \end{minipage}
}

\def\Ntwist{\unitlength.1em
  \begin{minipage}{15\unitlength}
    \begin{picture}(15,15)
     \put(0,7.5){\line(1,0){6}}
     \put(9,7.5){\line(1,0){6}}
      \put(4,5.5){\line(1,0){3}}
      \qbezier(7,5.5)(7,5.5)(9,9.5)
      \put(9, 9.5){\line(1,0){3}}
    \end{picture}
  \end{minipage}
}

\usetikzlibrary{decorations.markings}

\tikzset{->-/.style={decoration={
  markings,
  mark=at position .5 with {\arrow{>}}},postaction={decorate}}}

\tikzset{-<-/.style={decoration={
  markings,
  mark=at position .5 with {\arrow{<}}},postaction={decorate}}}

\begin{document}

\title[A topological interpretation of Viro's $\mathfrak{gl}(1\vert 1)$-Alexander polynomial]{A topological interpretation of Viro's $\mathfrak{gl}(1\vert 1)$-Alexander polynomial of a graph}
\author{Yuanyuan Bao}
\address{
Graduate School of Mathematical Sciences, University of Tokyo, 3-8-1 Komaba, Tokyo 153-8914, Japan
}
\email{bao@ms.u-tokyo.ac.jp}


\maketitle

\begin{abstract}
For an oriented trivalent graph $G$ without source or sink embedded in $S^3$, we prove that the $\mathfrak{gl}(1\vert 1)$-Alexander polynomial $\underline{\Delta}(G, c)$ defined by Viro satisfies a series of relations, which we call MOY-type relations in \cite{alex}. As a corollary we show that the Alexander polynomial $\Delta_{(G, c)}(t)$ studied in \cite{alex} coincides with $\underline{\Delta}(G, c)$ for a positive coloring $c$ of $G$, where $\Delta_{(G, c)}(t)$ is constructed from a certain regular covering space of the complement of $G$ in $S^3$ and it is the Euler characteristic of the Heegaard Floer homology of $G$ that we studied before. When $G$ is a plane graph, we provide a topological interpretation to the vertex state sum of $\underline{\Delta}(G, c)$ by considering a special Heegaard diagram of $G$ and the Fox calculus on the Heegaard surface.
\end{abstract}

\section{Introduction}
The Alexander polynomial of a knot was first studied by J. W. Alexander in 1920s. It is a well-known fact that it can be interpreted in several different ways. For example, it can be defined via the universal abelian covering space of the knot complement. There is a definition for it from the Burau representation of the braid group. By applying the representation theory of quantum groups (see J. Murakami \cite{MR1197048, MR1183500}, Kauffman and Saleur \cite{MR1133269}, Rozansky and Saleur \cite{MR1170953}, and Reshetikhin \cite{MR1223142}), people found that it is also a $\mathfrak{gl}(1\vert 1)$- and $\mathfrak{sl}(2)$-quantum invariant. In recent years this invariant has come to draw a lot of attention for the reason that it is the Euler characteristic of the knot Floer homology introduced by Ozsv{\'a}th-Szab{\'o} and Rasmussen independently. 

The colored Alexander polynomial of a link in $S^3$ was defined in \cite{MR1164114}, where the authors considered the $\mathcal{U}_q(\mathfrak{sl}(2))$-representation at $2n$-th root of unity. The classical Alexander polynomial corresponds to the $n=2$ case. Unlike the $n=2$ case, few results are known about the topological meaning of the colored Alexander polynomial for higher $n$'s. 

The extension of the Alexander polynomial to a trivalent graph as a quantum invariant is a natural thing to do, and has been studied many times in the literature. For instance, in \cite{MR2255851} Viro generalized the multi-variable Alexander polynomial (Conway function) to trivalent graphs equipped with admissible colorings and found a face state sum model for his generalization. His definition is based on the irreducible representations of quantum (super)groups $\mathcal{U}_q(\mathfrak{gl}(1\vert 1))$ and $\mathcal{U}_q(\mathfrak{sl}(2))$. In \cite{MR3073565} Costantino and Murakami defined the colored Alexander polynomial of a trivalent graph and constructed a face state sum model for this invariant by using the corresponding $6j$-symbols. Viro's $\mathcal{U}_q(\mathfrak{sl}(2))$-model corresponds to the $n=2$ case of Costantino and Murakami's theory.

In this paper, we provide a topological interpretation of Viro's $\mathfrak{gl}(1\vert 1)$-Alexander polynomial of a graph. 
For an oriented trivalent graph $G$ without source or sink embedded in $S^3$ and a positive coloring defined on it, we constructed a polynomial $\Delta_{(G, c)}(t)$ in \cite{alex} and studied a series of relations that it satisfies. The polynomial was defined by considering a certain Alexander module associated with a regular covering space of the complement of the graph in $S^3$. It was originally studied in \cite{bao}, where we studied the Heegaard Floer homology for a balanced bipartite graph with a proper orientation and showed that a multi-variable version of $\Delta_{(G, c)}(t)$ is the Euler characteristic of the homology. In this paper, we show that Viro's $\mathfrak{gl}(1\vert 1)$-Alexander polynomial $\underline{\Delta}(G, c)$ of a graph coincides with $\Delta_{(G, c)}(t)$. Precisely, in Section 3 we prove the following theorem.

\newtheorem*{maintheorem}{Theorem \ref{maintheorem}}
\begin{maintheorem}
For an oriented trivalent graph $G$ without source or sink, let $c$ be a coloring whose multiplicities are positive and weights are one. We have 
$$
\Delta_{(G, c)}(q^{-4}) =\displaystyle \frac{\prod_{\text{$v$: even type}} \{2c(v)\}_q }{(q^{2}-q^{-2})^{\vert V \vert-1}}\underline{\Delta}(G, c),
$$ where $\vert V \vert$ is the number of vertices in $G$.
\end{maintheorem}

We use two approaches to analyze the coincidence. The first approach is based on MOY-type relations. In \cite{alex} we showed that $\Delta_{(G,c)}(t)$ satisfies a series of relations, which we call MOY-type relations since they are analogous to MOY's relations in \cite{MR1659228}, and proved that these MOY-type relations characterize $\Delta_{(G,c)}(t)$. In Theorem \ref{moy} we prove that $\underline{\Delta}(G, c)$ satisfies an adapted version of the MOY-type relations. By comparing these relations with those for $\Delta_{(G,c)}(t)$, in Section 3 we prove Theorem \ref{maintheorem}.

It would be interesting to compare the relations in Theorem \ref{moy} with those in \cite{MR3470707}, which is based on quantum skew Howe duality. Apparently they are quite different from each other, since in \cite{MR3470707} the color of each edge stands for the degree of the exterior power of the defining representation of $\mathcal{U}_{q}(\mathfrak{gl}(1\vert 1))$. It is natural to expect a precise connection between them, which is unknown to the author.

The other approach is based on a special Heegaard diagram of $G$ when $G$ is a plane graph. In this case we show that the morphism around a trivalent vertex, which determines $\underline{\Delta}(G, c)$, can be obtained from the Fox calculus on the Heegaard diagram. The morphism was originally obtained by scaling the Clebsch-Gordan morphisms for irreducible $\mathcal{U}_{q}(\mathfrak{gl}(1\vert 1))$-modules of dimension $(1 \vert 1)$.

\vspace{5pt}\noindent{\bf Acknowledgements.} We would like to thank Hitoshi Murakami and Zhongtao Wu for helpful discussions and comments. We thank the anonymous referee for careful reading and many helpful suggestions.

\section{The MOY-type relations for Viro's $\mathfrak{gl}(1\vert 1)$-Alexander polynomial}
\subsection{Viro's $\mathfrak{gl}(1\vert 1)$-Alexander polynomial}
Viro \cite{MR2255851} defined a functor from the category of colored framed trivalent graph to the category of finite dimensional modules over the $q$-deformed universal enveloping algebra $\mathcal{U}_{q}(\mathfrak{gl}(1\vert 1))$. Using this functor, in Section 6 of \cite{MR2255851}, he constructed the $\mathfrak{gl}(1\vert 1)$-Alexander polynomial of a trivalent graph. Instead of recalling a full definition of the functor, we recall the definition and calculation of the polynomial.

\begin{rem}
Sartori's note \cite{MR3319619} is another commonly used reference about the structure and the finite-dimensional representations of $\mathcal{U}_{q}(\mathfrak{gl}(1\vert 1))$. Note that there exist minor differences of conventions between \cite{MR2255851} and \cite{MR3319619}. If we identify $E, G, Y, X$ in \cite{MR2255851} with $H_1+H_2, H_2, E, F$ in \cite{MR3319619} respectively, we see that the choices of coalgebras and the universal R-matrices are slightly different. In this paper, we follow Viro's conventions.
\end{rem}

For us, a graph will be an oriented trivalent graph without source or sink, embedded in $S^3$, possibly with loops and multi-edges.
Let $G$ be a graph, and
let $E$ be the set of edges and $V$ be the set of vertices of $G$. Consider a map which we call a coloring
\begin{eqnarray*}
c: E &\to& \mathbb{Z}\backslash \{0\}\oplus \mathbb{Z}\\
e &\mapsto& (j, J).
\end{eqnarray*} 
The first number $j$ is called the {\it multiplicity} and the second number $J$ is called the {\it weight}. 
Around a vertex, suppose the three edges adjacent to it are colored by $(j_1, J_1)$, $(j_2, J_2)$ and $(j_3, J_3)$. Let $\epsilon_i=-1$ if the $i$-th edge points toward the vertex and $\epsilon_i=1$ otherwise. The coloring $c$ needs to satisfy the following admissibility conditions. 
\begin{align*}
\text{Admissibility conditions} \hspace{2cm} \sum_{i=1}^{3} \epsilon_i j_i&=0,\hspace{7cm}\\
\sum_{i=1}^{3} \epsilon_i J_i&=-\prod_{i=1}^{3} \epsilon_i.
\end{align*}
%The Alexander polynomial $\underline{\Delta}(G, c)$ is defined for $G$ with the coloring $c$. 

The pair $(j, J)$ corresponds to two irreducible $\mathcal{U}_{q}(\mathfrak{gl}(1\vert 1))$-modules of dimension $(1 \vert 1)$, which are denoted by $(j, J)_{+}$ and $(j, J)_{-}$. These two modules are dual to each other. The module $(j, J)_{+}$ (resp. $(j, J)_{-}$) is generated by two elements $e_0$ (boson) and $e_1$ (fermion). For details of their definitions please see Appendix 1 of \cite{MR2255851}. 

A {\it framing} of $G$ is an oriented compact surface $F$ embedded in $S^3$ in which $G$ is sitting as a deformation retract. More precisely, in $F$ each vertex of $G$ is replaced by a disk where the vertex is the center, and each edge of $G$ is replaced by a strip $[0, 1]\times [0, 1]$ where $[0, 1]\times \{0, 1\}$ is attached to the boundaries of its adjacent vertex disks and $\{\frac{1}{2}\}\times [0, 1]$ is the given edge of $G$. See Fig. \ref{fig:e27} for an example.

\begin{figure}
	\centering
		\includegraphics[width=0.6\textwidth]{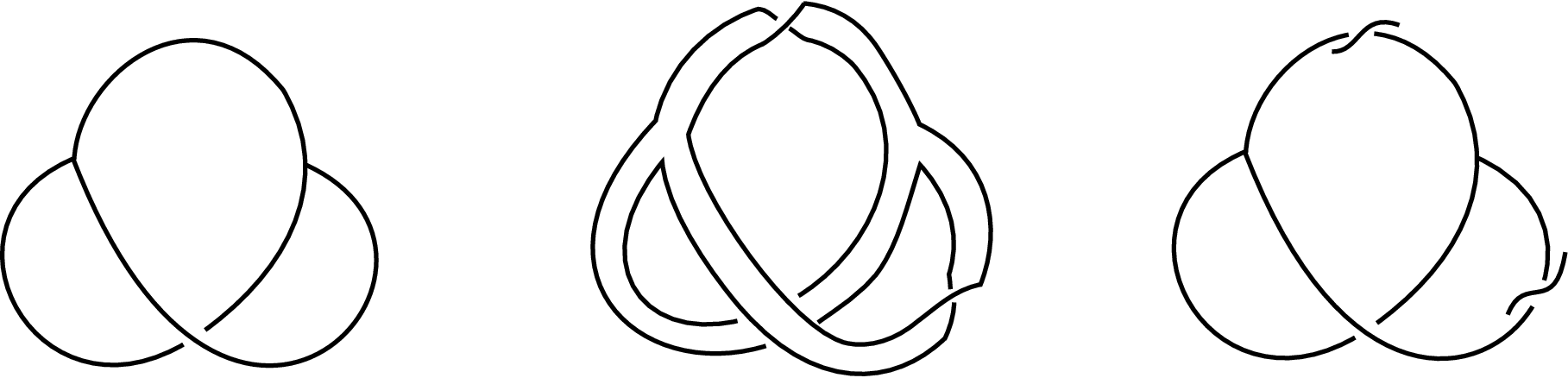}
	\caption{A framed graph and its graph diagram.}
	\label{fig:e27}
\end{figure}

A {\it framed graph} is a graph with a framing. By an {\it isotopy} of a framed graph we mean an isotopy of the graph in $S^{3}$ which extends to an isotopy of the framing. A graph diagram of $G$ in $\mathbb{R}^{2}$ can be equipped with a framing such that the tubular neighborhood of the graph diagram in $\mathbb{R}^{2}$ is an immersion of the framing. It is called the {\it blackboard framing}.

The difference between a generic framing and the blackboard framing can be represented on the diagram by introducing the half-twist symbols \Ptwist and \Ntwist, which mean a positive half twist and a negative half twist respectively at the fragment. Therefore we can use a graph diagram with \Ptwist and \Ntwist to represent a framed graph, as shown in Fig. \ref{fig:e27}. 

Now we review the definition of $\underline{\Delta}(G, c)$ for a framed graph $G$ with a coloring $c$. Choose 
a graph diagram of $G$ in $\mathbb{R}^2$. The diagram divides $\mathbb{R}^2$ into several regions, one of which is unbounded. Choose an edge of $G$ on the boundary of the unbounded region and cut the edge at a generic point. Suppose the color of the edge is $(j, J)$. Deform the graph diagram under isotopies of $\mathbb{R}^2$ to make it in a Morse position under a given orthogonal coordinate system of $\mathbb{R}^{2}$ so that the two endpoints created by cutting have heights zero and one and the critical points, the crossings, the half-twist symbols and the vertices of the diagram have different heights between zero and one. Namely after deformation the diagram can be divided into several pieces by horizontal lines so that each piece is a disjoint union of trivial vertical segments with one of the eight elements in Fig.~\ref{fig2}. Each piece connects a sequence of endpoints on its bottom to a sequence of endpoints on its top.

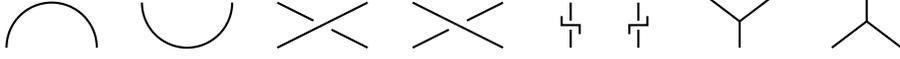
\begin{figure}
\begin{tikzpicture}[baseline=-0.65ex, thick, scale=0.6]
\draw (-1, 0) arc (0:180:1);
\draw (0, 1) arc (180:360:1);
\draw (3, 0) -- (5, 1);
\draw (5, 0) -- (4.2, 0.4);
\draw (3, 1) -- (3.8, 0.6);
\draw (8, 0) -- (6, 1);
\draw (8, 1) -- (7.2, 0.6);
\draw (6, 0) -- (6.8, 0.4);
\draw (9.5, 0) -- (9.5, 0.4);
\draw (9.5, 1) -- (9.5, 0.6);
\draw (9.3, 0.7) -- (9.3, 0.5) -- (9.7, 0.5) --(9.7, 0.3);
\draw (11, 0) -- (11, 0.4);
\draw (11, 1) -- (11, 0.6);
\draw (10.8, 0.3) -- (10.8, 0.5) -- (11.2, 0.5) --(11.2, 0.7);
\end{tikzpicture}\hspace{6mm}
\begin{tikzpicture}[baseline=-0.65ex, thick, scale=0.7]
\draw  (0, 0)  to (0,0.5);
\draw   (0,0.5)   to  (0.66,1);
\draw   (0,0.5)  to  (-0.66,1);
\end{tikzpicture}\hspace{6mm}
\begin{tikzpicture}[baseline=-0.65ex, thick, scale=0.7]
\draw  (0, 0.5)  to (0,1);
\draw   (0,0.5)   to  (0.66,0);
\draw   (0,0.5)   to  (-0.66,0);
\end{tikzpicture}
\caption{Critical points, crossings, half-twist symbols and vertices. Any orientations, without creating sources or sinks, are allowed.}

\label{fig2}
\end{figure}

Under Viro's functor, each piece, read from bottom to top,  is mapped to a morphism between tensor products of irreducible $\mathcal{U}_q(\mathfrak{gl}(1\vert 1))$-modules of dimension $(1 \vert 1)$. Suppose the sequence of endpoints for a given piece on the bottom (resp. top) is $(p_1, \cdots, p_k)$ for $k\geq 1$ (resp. $(q_1, \cdots, q_l)$ for $l\geq 1$) where the subindices respect the $x$-coordinates of the endpoints. Then $(p_1, \cdots, p_k)$ (resp. $(q_1, \cdots, q_l)$) corresponds to the tensor product $(j_1, J_1)_{\epsilon_1}\otimes \cdots \otimes (j_k, J_k)_{\epsilon_k}$ (resp. $(i_1, I_1)_{\epsilon_1}\otimes \cdots \otimes (i_l, I_l)_{\epsilon_l}$), where $(j_r, J_r)$ (resp. $(i_s, I_s)$) is the color of the edge containing $p_r$ (resp. $q_s$) and $\epsilon_r=+$ (resp. $\epsilon_s=+$) when the edge points upward and $\epsilon_r=-$ (resp. $\epsilon_s=-$) otherwise for $1\leq r\leq k$ (resp. $1\leq s\leq l$). Then the morphism is from $(j_1, J_1)_{\epsilon_1}\otimes \cdots \otimes (j_k, J_k)_{\epsilon_k}$ to 
$(i_1, I_1)_{\epsilon_1}\otimes \cdots \otimes (i_l, I_l)_{\epsilon_l}$. 

The morphism is defined in the language of Boltzmann weights. Simply speaking, each module $(j_r, J_r)_{\epsilon_r}$ (resp. $(i_s, I_s)_{\epsilon_s}$) has two generators $e_0$ (boson) and $e_1$ (fermion), and therefore 
$(j_1, J_1)_{\epsilon_1}\otimes \cdots \otimes (j_k, J_k)_{\epsilon_k}$ (resp. $(i_1, I_1)_{\epsilon_1}\otimes \cdots \otimes (i_l, I_l)_{\epsilon_l}$) is generated by $\{\otimes_{r=1}^{k}e_{\delta_r}\}_{\delta_r=0, 1}$ (resp. $\{\otimes_{s=1}^{l}e_{\delta_s}\}_{\delta_s=0, 1}$). The morphism is represented by a matrix under the above choice of generators, and the Boltzmann weights are the entries of the matrix.

The composition of two pieces (attaching them by identifying the top of the first piece with the bottom of the second piece) corresponds to the composition of their morphisms for $\mathcal{U}_q(\mathfrak{gl}(1\vert 1))$-modules. As a consequence, the graph diagram in Morse position with two endpoints of heights zero and one is mapped to a morphism from $(j, J)_+$ to $(j, J)_+$ (or $(j, J)_-$ to $(j, J)_-$ depending the orientation of $G$ at the endpoints), which acts as the multiplication of a rational function of $q$ (\cite[5.1.A]{MR2255851}). Recall that $(j, J)$ is the color of the edge which was cut. Then dividing the rational function by $q^{2j}-q^{-2j}$ we get $\underline{\Delta}(G, c)$. To calculate $\underline{\Delta}(G, c)$, we can take the following steps provided by Viro in \cite[Section 6.4]{MR2255851}.

%The projection divides $\mathbb{R}^2$ into several domains. Choose an edge of $G$ on the boundary of the unbounded domain and cut the edge at a generic point. 
%Suppose the edge where the graph is cut is colored by $(j, J)$. The resulting diagram with coloring $c$ corresponds to a morphism from $(j, J)_{+}$ to $(j, J)_{+}$, which as discussion in \cite{MR2255851} is a multiplication by a rational function on a variable $q$. Then $\underline{\Delta}^{c}(G)$ is defined by dividing the function by $q^{2j}-q^{-2j}$. 

\begin{table}
\begin{center}
\begin{tabular} {|c|c|c|c|c|} \hline
&\begin{tikzpicture}[baseline=-0.65ex, thick]
\draw [dotted] (0, 0) [<-] to [out=270,in=180] (0.5, -0.5) to [out=0,in=270] (1,0);
\draw (1.7, -0.3) node {$(j, J)$};
\draw (0.7, -1) node {$+1$};
\end{tikzpicture}&
\begin{tikzpicture}[baseline=-0.65ex, thick]
\draw  (0, 0) [<-] to [out=270,in=180] (0.5, -0.5) to [out=0,in=270] (1,0);
\draw (1.7, -0.3) node {$(j, J)$};
\draw (0.7, -1) node {$-q^{2j}$};
\end{tikzpicture}&
\begin{tikzpicture}[baseline=-0.65ex, thick]
\draw [dotted] (0, 0) to [out=270,in=180] (0.5, -0.5)  [->] to [out=0,in=270] (1,0);
\draw (1.7, -0.3) node {$(j, J)$};
\draw (0.7, -1) node {$q^{-2j}$};
\end{tikzpicture}
&
\begin{tikzpicture}[baseline=-0.65ex, thick]
\draw (0, 0) to [out=270,in=180] (0.5, -0.5) [->] to [out=0,in=270]  (1,0);
\draw (1.7, -0.3) node {$(j, J)$};
\draw (0.7, -1) node {$1$};
\end{tikzpicture}\\
\hline
&\begin{tikzpicture}[baseline=-0.65ex, thick]
\draw [dotted] (0, 0) [<-] to [out=90,in=180] (0.5, 0.5) to [out=0,in=90] (1,0);
\draw (1.7, 0.3) node {$(j, J)$};
\draw (0.7, -0.5) node {$+1$};
\end{tikzpicture}&
\begin{tikzpicture}[baseline=-0.65ex, thick]
\draw  (0, 0) [<-] to [out=90,in=180] (0.5, 0.5) to [out=0,in=90] (1,0);
\draw (1.7, 0.3) node {$(j, J)$};
\draw (0.7, -0.5) node {$-q^{-2j}$};
\end{tikzpicture}&
\begin{tikzpicture}[baseline=-0.65ex, thick]
\draw [dotted] (0, 0)  to [out=90,in=180] (0.5, 0.5) [->] to [out=0,in=90] (1,0);
\draw (1.7, 0.3) node {$(j, J)$};
\draw (0.7, -0.5) node {$q^{2j}$};
\end{tikzpicture}
&
\begin{tikzpicture}[baseline=-0.65ex, thick]
\draw  (0, 0)  to [out=90,in=180] (0.5, 0.5) [->] to [out=0,in=90] (1,0);
\draw (1.7, 0.3) node {$(j, J)$};
\draw (0.7, -0.5) node {$1$};
\end{tikzpicture}\\
\hline
\hline
& \begin{tikzpicture}[baseline=-0.65ex, thick, scale=0.6]
\draw [dotted] (0, -1) -- (0, 0);
\draw [dotted] (0, 0.3)  -- (0, 1);
\draw (0.2, 0.5) -- (0.2, 0.2) -- (-0.2, 0.1) -- (-0.2, -0.2);
\draw (1.2, 0.8) node {$(j, J)$};
\draw (0, -1.5) node {$q^{-jJ}$};
\end{tikzpicture}
&
\begin{tikzpicture}[baseline=-0.65ex, thick, scale=0.6]
\draw  (0, -1) -- (0, 0);
\draw  (0, 0.3)  -- (0, 1);
\draw (0.2, 0.5) -- (0.2, 0.2) -- (-0.2, 0.1) -- (-0.2, -0.2);
\draw (1.2, 0.8) node {$(j, J)$};
\draw (0, -1.5) node {$q^{-jJ}$};
\end{tikzpicture}
&
\begin{tikzpicture}[baseline=-0.65ex, thick, scale=0.6]
\draw [dotted] (0, -1) -- (0, 0);
\draw [dotted] (0, 0.3)  -- (0, 1);
\draw (-0.2, 0.5) -- (-0.2, 0.2) -- (0.2, 0.1) -- (0.2, -0.2);
\draw (1.2, 0.8) node {$(j, J)$};
\draw (0, -1.5) node {$q^{jJ}$};
\end{tikzpicture}
&
\begin{tikzpicture}[baseline=-0.65ex, thick, scale=0.6]
\draw  (0, -1) -- (0, 0);
\draw  (0, 0.3)  -- (0, 1);
\draw (-0.2, 0.5) -- (-0.2, 0.2) -- (0.2, 0.1) -- (0.2, -0.2);
\draw (1.2, 0.8) node {$(j, J)$};
\draw (0, -1.5) node {$q^{jJ}$};
\end{tikzpicture}
\\
\hline
\hline
\vspace{-2mm}&&&&\\
     & \begin{tikzpicture}[baseline=-0.65ex, thick]
\draw [dotted] (0, 0)  to (1,1);
\draw  [dotted] (1,0)  to  (0,1);
\end{tikzpicture}
     & \begin{tikzpicture}[baseline=-0.65ex, thick]
\draw [dotted] (0, 0)  to (1,1);
\draw  (1,0)  to  (0,1);
\end{tikzpicture}
      & \begin{tikzpicture}[baseline=-0.65ex, thick]
\draw  (0, 0)  to (1,1);
\draw  [dotted] (1,0)  to  (0,1);
\end{tikzpicture}
& \begin{tikzpicture}[baseline=-0.65ex, thick]
\draw  (0, 0)  to (1,1);
\draw  (1,0)  to  (0,1);
\end{tikzpicture}
       \\ \hline 
 \vspace{-2mm}     &&&&\\
\begin{tikzpicture}[baseline=-0.65ex, thick]
\draw  (0, 0) [->]  to (1,1);
\draw  (1,0)  to  (0.6,0.4);
\draw  (0.4,0.6) [->]  to  (0,1);
\draw (1.5, 0) node {$(j, J)$};
\draw (-0.5, 0) node {$(i, I)$};
\end{tikzpicture}
 & $q^{-iJ-jI}q^{i+j}$   & $q^{-iJ-jI}q^{j-i}$ & $q^{-iJ-jI}q^{i-j}$ &$-q^{-iJ-jI}q^{-i-j}$ \\ 
\hline 
 \vspace{-2mm}     &&&&\\
\begin{tikzpicture}[baseline=-0.65ex, thick]
\draw  (0.6, 0.6) [->]  to (1,1);
\draw  (1,0) [->] to  (0,1);
\draw  (0,0)  to  (0.4,0.4);
\draw (1.5, 0) node {$(j, J)$};
\draw (-0.5, 0) node {$(i, I)$};
\end{tikzpicture}
 & $q^{iJ+jI}q^{-i-j}$   & $q^{iJ+jI}q^{i-j}$ & $q^{iJ+jI}q^{j-i}$ &$-q^{iJ+jI}q^{i+j}$ \\ 
\hline 
\hline
\vspace{-2mm}&&&&\\
     & \begin{tikzpicture}[baseline=-0.65ex, thick]
\draw (0, 0)  to (0.5,0.5);
\draw [dotted] (0.5, 0.5)  to (1,1);
\draw (0.5, 0.5) to (0, 1);
\draw  [dotted] (1,0)  to  (0.5,0.5);
\end{tikzpicture}
     & \begin{tikzpicture}[baseline=-0.65ex, thick]
\draw [dotted] (0, 0)  to (0.5,0.5);
\draw  (0.5, 0.5)  to (1,1);
\draw [dotted] (0.5, 0.5) to (0, 1);
\draw  (1,0)  to  (0.5,0.5);
\end{tikzpicture}
      & \begin{tikzpicture}[baseline=-0.65ex, thick]
\draw (0, 0)  to (0.5,0.5);
\draw [dotted]  (0.5, 0.5)  to (1,1);
\draw [dotted] (0.5, 0.5) to (0, 1);
\draw  (1,0)  to  (0.5,0.5);
\end{tikzpicture}
& \begin{tikzpicture}[baseline=-0.65ex, thick]
\draw  [dotted] (0, 0)  to (0.5,0.5);
\draw  (0.5, 0.5)  to (1,1);
\draw   (0.5, 0.5) to (0, 1);
\draw  [dotted] (1,0)  to  (0.5,0.5);
\end{tikzpicture}
       \\ \hline 
 \vspace{-2mm}     &&&&\\
\begin{tikzpicture}[baseline=-0.65ex, thick]
\draw  (0, 0) [->]  to (1,1);
\draw  (1,0)  to  (0.6,0.4);
\draw  (0.4,0.6) [->]  to  (0,1);
\draw (1.5, 0) node {$(j, J)$};
\draw (-0.5, 0) node {$(i, I)$};
\end{tikzpicture}
 & $0$   & $\displaystyle \frac{q^{4i}-1}{q^{iJ+jI+i+j}}$ & $0$ &$0$ \\ 
\hline 
 \vspace{-2mm}     &&&&\\
\begin{tikzpicture}[baseline=-0.65ex, thick]
\draw  (0.6, 0.6) [->]  to (1,1);
\draw  (1,0) [->] to  (0,1);
\draw  (0,0)  to  (0.4,0.4);
\draw (1.5, 0) node {$(j, J)$};
\draw (-0.5, 0) node {$(i, I)$};
\end{tikzpicture}
 & $\displaystyle \frac{1-q^{4j}}{q^{-iJ-jI+i+j}}$   & $0$ & $0$ &$0$ \\ 
\hline
\end{tabular}
\vspace{3mm}
\caption{Boltzmann weights for critical points, half-twist symbols and two types of crossings from Viro's Table 1.}
\label{viro1}
\end{center}
\end{table}

\begin{prop}[vertex state sum representation in \cite{MR2255851}]
\label{sum} 
$\underline{\Delta}(G, c)$ can be calculated as follows.
\begin{enumerate}
\item Choose an orthogonal coordinate system for $\mathbb{R}^2$ and consider a graph diagram of $G$ in a Morse position in $\mathbb{R}^2$. Choose a generic point $\delta$ on the leftmost edge of $G$, which we call initial point. Suppose the color of the edge is $(j, J)$. 
\item Assign the generator $e_0$ to the edge with initial point, and assign $e_0$ or $e_1$ to each of the other edges. Such an assignment is called a state. In a state, if an edge is assigned with $e_0$ (resp. $e_1$), we represent it by a dotted (resp. solid) line, as in Fig. \ref{fig3}.
\item For each state, take the product of the Boltzmann weights at the critical points of the diagram, the crossings, the half-twist symbols, and the vertices of the graph. The Boltzmann weights for all possible assignments are defined in Tables 1 and 2 of \cite{MR2255851}. See Tables \ref{viro1} and \ref{viro2} for part of the data which we will use later.
\item Take the sum of the products over all the states.  Multiplying the sum by $ \frac{q^{(-1)^{\theta}2j}}{q^{2j}-q^{-2j}}$ we get $\underline{\Delta}(G, c)$, where $\theta=0$ (resp. $1$) if the edge segment around $\delta$ points downward (resp. upward). 
\end{enumerate}

\end{prop}

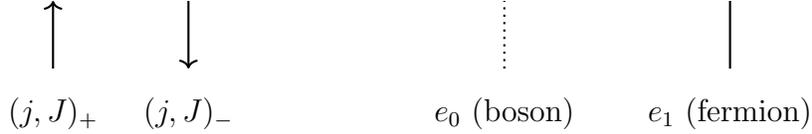
\begin{figure}
\centering
\begin{tikzpicture}[baseline=-0.65ex, thick, scale=0.6]
\draw (0, -1)  [->] to (0, 0.5);
\draw (3, -1)  [<-] to (3, 0.5);
\draw (0,-2) node {$(j, J)_{+}$};
\draw (3,-2) node {$(j, J)_{-}$};
\draw [dotted] (10, -1)  -- (10, 0.5);
\draw (15, -1)  -- (15, 0.5);
\draw (10,-2) node {$e_0$ (boson)};
\draw (15,-2) node {$e_1$ (fermion)};
\end{tikzpicture}
\caption{Under the coloring $c$, each edge corresponds to an irreducible $\mathcal{U}_{q}(\mathfrak{gl}(1\vert 1))$-module. In a state, if an edge is assigned with $e_0$ (resp. $e_1$), we represent it by a dotted (resp. solid) line.}
\label{fig3}
\end{figure}

\begin{rem}
Viro defined a multi-variable version of the Alexander polynomial in \cite[Section 6]{MR2255851}, where each edge is colored by an irreducible module over a subalgebra $U^1$ of $\mathcal{U}_{q}(\mathfrak{gl}(1\vert 1))$. Here we consider a single variable version. The Boltzmann weights for the single variable version are given in Tables 1 and 2 of \cite{MR2255851}. 
\end{rem}

\begin{table}
\begin{tabular}{|c|c|c|c|c|} \hline
 \vspace{-3mm}     &&&&\\
& \begin{tikzpicture}[baseline=-0.65ex, thick]
\draw [dotted] (0, -0.5)  to (0,0);
\draw  [dotted] (0,0)   to  (0.66,0.5);
\draw  [dotted] (0,0)   to  (-0.66,0.5);
\end{tikzpicture}
&
\begin{tikzpicture}[baseline=-0.65ex, thick]
\draw [dotted] (0, -0.5)  to (0,0);
\draw   (0,0)   to  (0.66,0.5);
\draw   (0,0)   to  (-0.66,0.5);
\end{tikzpicture}
&
\begin{tikzpicture}[baseline=-0.65ex, thick]
\draw  (0, -0.5)  to (0,0);
\draw  [dotted] (0,0)   to  (0.66,0.5);
\draw   (0,0)   to  (-0.66,0.5);
\end{tikzpicture}
&
\begin{tikzpicture}[baseline=-0.65ex, thick]
\draw  (0, -0.5)  to (0,0);
\draw   (0,0)   to  (0.66,0.5);
\draw  [dotted] (0,0)   to  (-0.66,0.5);
\end{tikzpicture}\\ 
 \vspace{-3mm}     &&&&\\
 \hline
\begin{tikzpicture}[baseline=-0.65ex, thick]
\draw  (0, -0.5) [->-] to (0,0);
\draw   (0,0) [->]  to  (0.66,0.5);
\draw  (0,0) [->]  to  (-0.66,0.5);
\draw (-0.9, 0.5) node {$i$};
\draw (0.9, 0.5) node {$j$};
\draw (0.3, -0.5) node {$k$};
\end{tikzpicture} &$q^{2k}-q^{-2k}$&$0$&$(q^{2i}-q^{-2i})q^{-2j}$&$q^{2j}-q^{-2j}$\\
\hline
\begin{tikzpicture}[baseline=-0.65ex, thick]
\draw  (0, -0.5) [->-] to (0,0);
\draw   (0,0) [-<-]  to  (0.66,0.5);
\draw  (0,0) [->]  to  (-0.66,0.5);
\draw (-0.9, 0.5) node {$i$};
\draw (0.9, 0.5) node {$j$};
\draw (0.3, -0.5) node {$k$};
\end{tikzpicture} &$1$&$-q^{2i}$&$1$&$0$\\
\hline
\begin{tikzpicture}[baseline=-0.65ex, thick]
\draw  (0, -0.5) [<-] to (0,0);
\draw   (0,0) [-<-]  to  (0.66,0.5);
\draw  (0,0) [->]  to  (-0.66,0.5);
\draw (-0.9, 0.5) node {$i$};
\draw (0.9, 0.5) node {$j$};
\draw (0.3, -0.5) node {$k$};
\end{tikzpicture} &$q^{2j}-q^{-2j}$&$-(q^{2i}-q^{-2i})q^{2j}$&$0$&$q^{2k}-q^{-2k}$\\
\hline
\begin{tikzpicture}[baseline=-0.65ex, thick]
\draw  (0, -0.5) [->-] to (0,0);
\draw   (0,0) [->]  to  (0.66,0.5);
\draw  (0,0) [-<-]  to  (-0.66,0.5);
\draw (-0.9, 0.5) node {$i$};
\draw (0.9, 0.5) node {$j$};
\draw (0.3, -0.5) node {$k$};
\end{tikzpicture} &$q^{-2i}$&$1$&$0$&$1$\\
\hline
\begin{tikzpicture}[baseline=-0.65ex, thick]
\draw  (0, -0.5) [<-] to (0,0);
\draw   (0,0) [->]  to  (0.66,0.5);
\draw  (0,0) [-<-]  to  (-0.66,0.5);
\draw (-0.9, 0.5) node {$i$};
\draw (0.9, 0.5) node {$j$};
\draw (0.3, -0.5) node {$k$};
\end{tikzpicture} &$(q^{2i}-q^{-2i})q^{-2j}$&$q^{2j}-q^{-2j}$&$q^{2k}-q^{-2k}$&$0$\\
\hline
\begin{tikzpicture}[baseline=-0.65ex, thick]
\draw  (0, -0.5) [<-] to (0,0);
\draw   (0,0) [-<-]  to  (0.66,0.5);
\draw  (0,0) [-<-]  to  (-0.66,0.5);
\draw (-0.9, 0.5) node {$i$};
\draw (0.9, 0.5) node {$j$};
\draw (0.3, -0.5) node {$k$};
\end{tikzpicture} &$1$&$0$&$1$&$q^{-2i}$\\
\hline
\hline
 \vspace{-3mm}     &&&&\\
& \begin{tikzpicture}[baseline=-0.65ex, thick]
\draw [dotted] (0, 0)  to (0,0.5);
\draw  [dotted] (0,0)   to  (0.66,-0.5);
\draw [dotted]  (0,0)   to  (-0.66,-0.5);
\end{tikzpicture}
&
\begin{tikzpicture}[baseline=-0.65ex, thick]
\draw [dotted] (0, 0)  to (0,0.5);
\draw   (0,0)   to  (0.66,-0.5);
\draw   (0,0)   to  (-0.66,-0.5);
\end{tikzpicture}
&
\begin{tikzpicture}[baseline=-0.65ex, thick]
\draw  (0, 0)  to (0,0.5);
\draw  [dotted] (0,0)   to  (0.66,-0.5);
\draw   (0,0)   to  (-0.66,-0.5);
\end{tikzpicture}
&
\begin{tikzpicture}[baseline=-0.65ex, thick]
\draw  (0, 0)  to (0,0.5);
\draw   (0,0)   to  (0.66,-0.5);
\draw [dotted]  (0,0)   to  (-0.66,-0.5);
\end{tikzpicture}\\ 
 \vspace{-3mm}     &&&&\\
 \hline
\begin{tikzpicture}[baseline=-0.65ex, thick]
\draw  (0, 0) [->] to (0,0.5);
\draw   (0,0) [-<-]  to  (0.66,-0.5);
\draw   (0,0) [-<-]  to  (-0.66,-0.5);
\draw (-0.9, -0.5) node {$i$};
\draw (0.9, -0.5) node {$j$};
\draw (0.3, 0.5) node {$k$};
\end{tikzpicture} &$1$&$0$&$1$&$q^{2i}$\\
\hline
\begin{tikzpicture}[baseline=-0.65ex, thick]
\draw  (0, 0) [->] to (0,0.5);
\draw   (0,0) [->]  to  (0.66,-0.5);
\draw   (0,0) [-<-]  to  (-0.66,-0.5);
\draw (-0.9, -0.5) node {$i$};
\draw (0.9, -0.5) node {$j$};
\draw (0.3, 0.5) node {$k$};
\end{tikzpicture} &$(q^{2i}-q^{-2i})q^{2j}$&$q^{2j}-q^{-2j}$&$q^{2k}-q^{-2k}$&$0$\\
\hline
\begin{tikzpicture}[baseline=-0.65ex, thick]
\draw  (0, 0) [-<-] to (0,0.5);
\draw   (0,0) [->]  to  (0.66,-0.5);
\draw   (0,0) [-<-]  to  (-0.66,-0.5);
\draw (-0.9, -0.5) node {$i$};
\draw (0.9, -0.5) node {$j$};
\draw (0.3, 0.5) node {$k$};
\end{tikzpicture} &$q^{2i}$&$1$&$0$&$1$\\
\hline
\begin{tikzpicture}[baseline=-0.65ex, thick]
\draw  (0, 0) [->] to (0,0.5);
\draw   (0,0) [-<-]  to  (0.66,-0.5);
\draw   (0,0) [->]  to  (-0.66,-0.5);
\draw (-0.9, -0.5) node {$i$};
\draw (0.9, -0.5) node {$j$};
\draw (0.3, 0.5) node {$k$};
\end{tikzpicture} &$q^{2j}-q^{-2j}$&$-(q^{2i}-q^{-2i})q^{-2j}$&$0$&$q^{2k}-q^{-2k}$\\
\hline
\begin{tikzpicture}[baseline=-0.65ex, thick]
\draw  (0, 0) [-<-] to (0,0.5);
\draw   (0,0) [-<-]  to  (0.66,-0.5);
\draw   (0,0) [->]  to  (-0.66,-0.5);
\draw (-0.9, -0.5) node {$i$};
\draw (0.9, -0.5) node {$j$};
\draw (0.3, 0.5) node {$k$};
\end{tikzpicture} &$1$&$-q^{-2i}$&$1$&$0$\\
\hline
\begin{tikzpicture}[baseline=-0.65ex, thick]
\draw  (0, 0) [-<-] to (0,0.5);
\draw   (0,0) [->]  to  (0.66,-0.5);
\draw   (0,0) [->]  to  (-0.66,-0.5);
\draw (-0.9, -0.5) node {$i$};
\draw (0.9, -0.5) node {$j$};
\draw (0.3, 0.5) node {$k$};
\end{tikzpicture} &$q^{2k}-q^{-2k}$&$0$&$(q^{2i}-q^{-2i})q^{2j}$&$q^{2j}-q^{-2j}$\\
\hline
\end{tabular}
\vspace{3mm}
\caption{The Boltzmann weights at a vertex from Viro's Table 2.}
\label{viro2}
\end{table}

\subsection{MOY-type relations for $\underline{\Delta}(G, c)$.}
In this section, we discuss a series of relations that $\underline{\Delta}(G, c)$ satisfies, which are adapted versions of the MOY-type relations in \cite{alex}. These relations were inspired by Murakami-Ohtsuki-Yamada's work in \cite{MR1659228}, where they provided a graphical definition for the $\mathcal{U}_q(\mathfrak{sl}_n)$-polynomial invariants of a link for any integer $n\geq 2$.

Before stating the relations, we make the following agreements.
\begin{itemize}
\item If several diagrams are involved in one relation, they are identical outside the local diagrams drawn in the relation. 
\item If the orientations of some edges are omitted in a relation, they may be recovered in any consistent way that does not create source or sink.
\item If the multiplicities or weights are omitted in a relation on some edges, they may be recovered in any consistent way that respects the admissibility conditions. 
\item In the definition of the coloring $c$, the multiplicity of an edge is not allowed to be zero. In the following relations, if zero appears as the multiplicity of an edge which is entirely included in a local diagram, we formally define the Alexander polynomial as below, where $\{k\}_q=q^k-q^{-k}$ for $k\in \mathbb{Z}$. 

\begin{align}
\left(\begin{tikzpicture}[baseline=-0.65ex, thick, scale=0.8]
\draw (0,-1)  to (0, 1);
\draw (1,-0.33) [->-] to (0, 0.33);
\draw (1, -1) to (1,1);
\draw (0,1.25) node {$i$};
\draw (0,-1.25) node {$i$};
\draw (1,1.25) node {$j$};
\draw (0.9,-1.25) node {$j$};
\draw (0.5,0.5) node {$0$};
%\draw (1, -0.33) node[circle,fill,inner sep=1pt]{};
%\draw (0, 0.33) node[circle,fill,inner sep=1pt]{};
\end{tikzpicture}\right):=\{2j\}_q\cdot
\left(\begin{tikzpicture}[baseline=-0.65ex, thick, scale=0.8]
\draw (0,-1)  to (0, 1);
%\draw (1,-0.33) [->-] to (0, 0.33);
\draw (1, -1) to (1,1);
\draw (0,-1.25) node {$i$};
\draw (0.9,-1.25) node {$j$};
%\draw (1, -0.33) node[circle,fill,inner sep=1pt]{};
%\draw (0, 0.33) node[circle,fill,inner sep=1pt]{};
\end{tikzpicture}\right),\quad
\left(\begin{tikzpicture}[baseline=-0.65ex, thick, scale=0.8]
\draw (0,-1)  to (0, 1);
\draw (1,0.33) [-<-] to (0, -0.33);
\draw (1, -1) to (1,1);
\draw (0,1.25) node {$i$};
\draw (0,-1.25) node {$i$};
\draw (1,1.25) node {$j$};
\draw (0.9,-1.25) node {$j$};
\draw (0.5,0.5) node {$0$};
%\draw (1, -0.33) node[circle,fill,inner sep=1pt]{};
%\draw (0, 0.33) node[circle,fill,inner sep=1pt]{};
\end{tikzpicture}\right):=\{2i\}_q\cdot
\left(\begin{tikzpicture}[baseline=-0.65ex, thick, scale=0.8]
\draw (0,-1)  to (0, 1);
%\draw (1,-0.33) [->-] to (0, 0.33);
\draw (1, -1) to (1,1);
\draw (0,-1.25) node {$i$};
\draw (0.9,-1.25) node {$j$};
%\draw (1, -0.33) node[circle,fill,inner sep=1pt]{};
%\draw (0, 0.33) node[circle,fill,inner sep=1pt]{};
\end{tikzpicture}\right).
\label{eq1} \tag{0}
\end{align}

\end{itemize}

In the proof of the following theorem, it is convenient to know the following identities, which can be verified by direct calculations. 

\begin{lemma}
\label{identity}
For $a, b, c\in \mathbb{Z}$, we have
\begin{align}
\{a+b\}_q=q^{-a}\{b\}_q+q^b\{a\}_q, \label{eq1} \tag{1}\\
\{a+b\}_q\{a-b\}_q=\{a\}_q^2-\{b\}_q^2, \label{eq2} \tag{2} \\
\{a\}_q\{b+c\}_q=\{a-c\}_q\{b\}_q+\{a+b\}_q\{c\}_q. \label{eq3} \tag{3}
\end{align}
\end{lemma}

\vspace{2mm}

\begin{theo}
\label{moy}
Viro's $\mathfrak{gl}(1\vert 1)$-Alexander polynomial $\underline{\Delta}(G, c)$ satisfies the following relations, where $(D)$ represents $\underline{\Delta}(D, c)$.
\begin{align*}\label{moy}
%%%%%%%%%%%%%%
 (i) \quad
&\left(\begin{tikzpicture}[baseline=-0.65ex, thick, scale=0.6]
\draw (0,0) circle (1);
\draw (1.5,0) node {$i$};
\end{tikzpicture}\right)
=\frac{1}{\{2i\}_q}.\\
%%%%%%%%%%%%%%
(ii) \quad
 & \text{$(D)=0$ if $D$ is a disconnected diagram.}\\
%%%%%%%%%%%%%%%
(iii) \quad
&\left(\begin{tikzpicture}[baseline=-0.65ex, thick, scale=0.6]
\draw (0, -1) -- (0, 0);
\draw (0, 0.3)  -- (0, 1) node[above]{$(i, I)$};
\draw (0.2, 0.5) -- (0.2, 0.2) -- (-0.2, 0.1) -- (-0.2, -0.2);
\end{tikzpicture}\right)
= q^{-iI} \cdot \left(\begin{tikzpicture}[baseline=-0.65ex, thick, scale=0.5]
\draw (0, -1)  -- (0, 1) node[above]{$(i, I)$};
\end{tikzpicture}\right), \quad
\left(\begin{tikzpicture}[baseline=-0.65ex, thick, scale=0.6]
\draw (0, -1) -- (0, 0);
\draw (0, 0.3)  -- (0, 1) node[above]{$(i, I)$};
\draw (-0.2, 0.5) -- (-0.2, 0.2) -- (0.2, 0.1) -- (0.2, -0.2);
\end{tikzpicture}\right)
= q^{iI} \cdot \left(\begin{tikzpicture}[baseline=-0.65ex, thick, scale=0.5]
\draw (0, -1)  -- (0, 1) node[above]{$(i, I)$};
\end{tikzpicture}\right).\\
%%%%%%%%%%%%%%%%
(iv) \quad
&\left(\begin{tikzpicture}[baseline=-0.65ex,thick,scale=0.6]
\draw [->-] (0.5,-1.5) node[below]{$i$}-- (0.5,-0.6);
\draw (0.5, -0.6) to [out=0,in=0] (0.5,0.6) ;
\draw (0.5, -0.6) to [out=180,in=180] (0.5,0.6);
\draw [->](0.5,0.6) -- (0.5,1.5) node[above]{$i$};
%\draw (0.5,-0.6) node[circle,fill,inner sep=1pt]{};
%\draw (0.5,0.6) node[circle,fill,inner sep=1pt]{};
\end{tikzpicture}\right) = \{2i\}_q\cdot
\left(\begin{tikzpicture}[baseline=-0.65ex,thick,scale=0.6]
\draw [->](0,-1.5) -- (0, 1.5);
\draw (0.5 , 0) node {$i$};
\draw (0.25, 0) node {};
\end{tikzpicture}\right).\\
%%%%%%%%%%%%%%%%
(v)
\quad
& \left(\begin{tikzpicture}[baseline=-0.65ex, thick, scale=1.2]
\draw (0,-1) [->-] to  (0.5,-0.5);
\draw (0.5, -0.5) [->-] to  (1,-0.5) ;
\draw (1, -0.5) [->] to  (1.5,-1);
\draw (0.5,-0.5) [-<-] to (0.5,0.5);
\draw (1,-0.5) [->-] to  (1,0.5);
\draw (0,1) [<-] to (0.5,0.5);
\draw (0.5, 0.5) [-<-] to  (1,0.5);
\draw (1, 0.5) [-<-]  to  (1.5,1);
\draw (0, -1.25) node {$j$};
\draw (1.5, -1.25) node {$i$};
\draw (0.2, 0) node {$i$};
\draw (0, 1.25) node {$j$};
\draw (1.5, 1.25) node {$i$};
\draw (1.3, 0) node {$j$};
\draw (0.75,-0.75) node {$i+j$};
\draw (0.75,0.75) node {$i+j$};
%\draw (0.5, -0.5) node[circle,fill,inner sep=1pt]{};
%\draw (1, -0.5) node[circle,fill,inner sep=1pt]{};
%\draw (0.5, 0.5) node[circle,fill,inner sep=1pt]{};
%\draw (1, 0.5) node[circle,fill,inner sep=1pt]{};
\end{tikzpicture}\right)
=\{2i+2j\}_q\cdot
\left(\begin{tikzpicture}[baseline=-0.65ex, thick, scale=1.2]
\draw (0,-1) [->-] to  (0.5,-0.5);
\draw (0.5, -0.5)  [->]  to  (1,-1);
\draw (0.5, 0.5)    [->] to (0,1);
\draw (0.5,0.5) [-<-] to  (1,1);
\draw (0.5,-0.5) [-<-] to  (0.5,0.5);
\draw (0, -1.25) node {$j$};
\draw (1, -1.25) node {$i$};
\draw (1.2, 0) node {$i-j$};
\draw (0, 1.25) node {$j$};
\draw (1, 1.25) node {$i$};
%\draw (0.75, -0.5) node[circle,fill,inner sep=1pt]{};
%\draw (0.75, 0.5) node[circle,fill,inner sep=1pt]{};
\end{tikzpicture}\right)
+\{2j\}_q^2\cdot
\left(\begin{tikzpicture}[baseline=-0.65ex, thick, scale=1.2]
\draw (0,-1) [->] to (0,1);
\draw (1,-1) [<-] to (1,1);
\draw (-0,-1.25) node {$j$};
\draw (1, -1.25) node {$i$};
\end{tikzpicture}\right). \\
%%%%%%%%%%%%%%%
(vi)
\quad
& \left(\begin{tikzpicture}[baseline=-0.65ex,thick, scale=0.9]
\draw (0,-1.5) [->-] to (0,-0.5);
\draw (0,-0.5) [->-] to  (-1,0.5);
\draw (-1,0.5) [->] to  (-2,1.5);
\draw (-1,0.5) [->] to  (-0.5,1.5);
\draw (0,-0.5) [->] to  (1,1.5);
\draw (-1.2,-1) node {$i+j+k$};
\draw (1, 2) node {$k$};
\draw (-1.4,-0.1) node {$i+j$};
\draw (-2, 2) node {$i$};
\draw (-0.5,2) node {$j$};
%\draw (0, -0.5) node[circle,fill,inner sep=1pt]{};
%\draw (-1, 0.5) node[circle,fill,inner sep=1pt]{};
\end{tikzpicture}\right)=\frac{\{2i+2j\}_q}{\{2j+2k\}_q}\cdot
\left(\begin{tikzpicture}[baseline=-0.65ex,thick, scale=0.9]
\draw (0,-1.5) [->-]-- (0,-0.5);
\draw (0,-0.5) [->] to  (-1,1.5);
\draw (1,0.5) [->] to  (2,1.5);
\draw (0,-0.5) [->-] to  (1,0.5);
\draw (1, 0.5) [->] to  (0.5,1.5);
\draw (1.3,-1) node {$i+j+k$};
\draw (2, 2) node {$k$};
\draw (1.5,0) node {$j+k$};
\draw (-1, 2) node {$i$};
\draw (0.5,2) node {$j$};
%\draw (0, -0.5) node[circle,fill,inner sep=1pt]{};
%\draw (1, 0.5) node[circle,fill,inner sep=1pt]{};
\end{tikzpicture}\right), \\
\quad
& \left(\begin{tikzpicture}[baseline=-0.65ex,thick, scale=0.9]
\draw (1,0.5) [->]-- (1,1.5);
\draw (2,-1.5) [->-] to  (1,0.5);
\draw (0,-0.5) [->-] to (1,0.5);
\draw (-1,-1.5) [->-] to  (0,-0.5);
\draw (0.5, -1.5) [->-] to  (0,-0.5);
\draw (-0.3,1) node {$i+j+k$};
\draw (2, -2) node {$k$};
\draw (-0.5,0.1) node {$i+j$};
\draw (-1, -2) node {$i$};
\draw (0.5,-2) node {$j$};
%\draw (1, 0.5) node[circle,fill,inner sep=1pt]{};
%\draw (0, -0.5) node[circle,fill,inner sep=1pt]{};
\end{tikzpicture}\right)=
\left(\begin{tikzpicture}[baseline=-0.65ex,thick, scale=0.9]
\draw (0,0.5) [->]-- (0,1.5);
\draw (-1,-1.5) [->-]  to  (0,0.5);
\draw (1,-0.5) [->-] to (0,0.5);
\draw (0.5,-1.5) [->-] to  (1,-0.5);
\draw (2, -1.5) [->-] to  (1,-0.5);
\draw (1.3,1) node {$i+j+k$};
\draw (2, -2) node {$k$};
\draw (1.4,0.2) node {$j+k$};
\draw (-1, -2) node {$i$};
\draw (0.5,-2) node {$j$};
%\draw (0, 0.5) node[circle,fill,inner sep=1pt]{};
%\draw (1, -0.5) node[circle,fill,inner sep=1pt]{};
\end{tikzpicture}\right).\\
%%%%%%%%%%%%%%%%%%
(vii)
\quad
& \left(\begin{tikzpicture}[baseline=-0.65ex, thick, scale=1.2]
\draw (0,-1) [->-] to (0, -0.33);
\draw (0, -0.33) [->-] to (0, 0.33);
\draw (0, 0.33) [->] to (0,1);
\draw (2,-1) [->-] to (2, -0.66);
\draw (2, -0.66) [->-] to (2, 0.66);
\draw (2, 0.66) [->] to (2,1);
\draw (0,-0.33) [-<-] to  (2,-0.66);
\draw (0,0.33) [->-] to (2,0.66);
\draw (0,-1.25) node {$i$};
\draw (2,-1.25) node {$j$};
\draw (0.5,0) node {$i+k$};
\draw (0.1,1.25) node {$i+k-l$};
\draw (1, -0.8) node {$k$};
\draw (1.5,0) node {$j-k$};
\draw (1, 0.8) node {$l$};
\draw (2,1.25) node {$j+l-k$};
%\draw (0, -0.33) node[circle,fill,inner sep=1pt]{};
%\draw (0, 0.33) node[circle,fill,inner sep=1pt]{};
%\draw (2, -0.66) node[circle,fill,inner sep=1pt]{};
%\draw (2, 0.66) node[circle,fill,inner sep=1pt]{};
\end{tikzpicture}\right)
= \frac{ \{2j\}_q\{2l\}_q}{ \{2j+2l-2k\}_q}\cdot
\left(
\begin{tikzpicture}[baseline=-0.65ex, thick, scale=1.2]
\draw (0,-1) [->-] to  (0.5,-0.33);
\draw (0.5, -0.33) [->-] to  (0.5,0.33);
\draw (0.5, 0.33) [->] to  (0,1);
\draw (1,-1) [->-] to (0.5,-0.33);
\draw (1,1) [<-] to  (0.5,0.33);
\draw (0, -1.25) node {$i$};
\draw (0.9,-1.25) node {$j$};
\draw (-0.5,1.25) node {$i+k-l$};
\draw (1.5,1.25) node {$j+l-k$};
\draw (1.2,0) node {$i+j$};
%\draw (0.5, -0.33) node[circle,fill,inner sep=1pt]{};
%\draw (0.5, 0.33) node[circle,fill,inner sep=1pt]{};
\end{tikzpicture}\right)\\
\quad & \hspace{4cm} + \frac{\{2i+2k-2l\}_q\{2j-2k\}_q}{\{2j+2l-2k\}_q}\cdot
\left(\begin{tikzpicture}[baseline=-0.65ex, thick, scale=1.2]
\draw (0,-1) [->-] to (0, 0.33);
\draw (0, 0.33) [->] to (0,1);
\draw (1,-1) [->-] to (1, -0.33);
\draw (1, -0.33) [->] to (1,1);
\draw (0,0.33) [-<-]  to  (1,-0.33);
\draw (-0.5,1.25) node {$i+k-l$};
\draw (0,-1.25) node {$i$};
\draw (1.5,1.25) node {$j+l-k$};
\draw (1,-1.25) node {$j$};
\draw (0.5,0.5) node {$k-l$};
%\draw (1, -0.33) node[circle,fill,inner sep=1pt]{};
%\draw (0, 0.33) node[circle,fill,inner sep=1pt]{};
\end{tikzpicture}\right).\\
%%%%%%%%%%%%%%
(viii) 
\quad
&\left(\begin{tikzpicture}[baseline=-0.65ex, thick, scale=0.55]
\draw (1,-1)   -- (0.2,-0.2);
\draw (-1, -1) [->] -- (1, 1) node[above]{$(i, I)$};
\draw (-0.2,0.2) [->] to (-1,1)  node[above]{$(j, J)$};
\end{tikzpicture}\right)
= \frac{q^{i-j-iJ-jI}}{\{2i\}_q}\cdot
\left(
\begin{tikzpicture}[baseline=-0.65ex, thick, scale=1.1]
\draw (0,-1) [->-] to  (0.5,-0.33);
\draw (0.5, -0.33) [->-] to  (0.5,0.33);
\draw (0.5, 0.33) [->] to  (0,1);
\draw (1,-1) [->-] to  (0.5,-0.33);
\draw (1,1) [<-] to  (0.5,0.33);
\draw (0, -1.25) node {$(i, I)$};
\draw (0.9,-1.25) node {$(j, J)$};
\draw (0,1.25) node {$(j, J)$};
\draw (1,1.25) node {$(i, I)$};
\draw (1,0) node {$i+j$};
%\draw (0.5, -0.33) node[circle,fill,inner sep=1pt]{};
%\draw (0.5, 0.33) node[circle,fill,inner sep=1pt]{};
\end{tikzpicture}\right) 
+  \frac{-q^{-i-j-iJ-jI}}{\{2i\}_q}\cdot
\left(\begin{tikzpicture}[baseline=-0.65ex, thick, scale=1.1]
\draw (0,-1) [->-] to (0, 0.33);
\draw (0, 0.33) [->] to (0,1);
\draw (1,-1) [->-] to (1, -0.33);
\draw (1, -0.33) [->] to (1,1);
\draw (0,0.33) [-<-]  to  (1,-0.33);
\draw (0,1.25) node {$(j, J)$};
\draw (0,-1.25) node {$(i, I)$};
\draw (1,1.25) node {$(i, I)$};
\draw (0.9,-1.25) node {$(j, J)$};
\draw (0.5,0.5) node {$j-i$};
%\draw (1, -0.33) node[circle,fill,inner sep=1pt]{};
%\draw (0, 0.33) node[circle,fill,inner sep=1pt]{};
\end{tikzpicture}\right)\\
 & \hspace{3cm}
 = \frac{q^{j-i-iJ-jI}}{\{2j\}_q}\cdot
\left(
\begin{tikzpicture}[baseline=-0.65ex, thick, scale=1.1]
\draw (0,-1) [->-] to  (0.5,-0.33);
\draw (0.5, -0.33) [->-] to  (0.5,0.33);
\draw (0.5, 0.33) [->] to  (0,1);
\draw (1,-1) [->-] to  (0.5,-0.33);
\draw (1,1) [<-] to  (0.5,0.33);
\draw (0, -1.25) node {$(i, I)$};
\draw (0.9,-1.25) node {$(j, J)$};
\draw (0,1.25) node {$(j, J)$};
\draw (1,1.25) node {$(i, I)$};
\draw (1,0) node {$i+j$};
%\draw (0.5, -0.33) node[circle,fill,inner sep=1pt]{};
%\draw (0.5, 0.33) node[circle,fill,inner sep=1pt]{};
\end{tikzpicture}\right)
+\frac{-q^{-i-j-iJ-jI}}{\{2j\}_q}\cdot
\left(\begin{tikzpicture}[baseline=-0.65ex, thick, scale=1.1]
\draw (0,-1) [->-] to (0, -0.33);
\draw (0, -0.33) [->] to (0,1);
\draw (1,-1) [->-] to (1, 0.33);
\draw (1, 0.33) [->] to (1,1);
\draw (0,-0.33) [->-]  to  (1,0.33);
\draw (0,1.25) node {$(j, J)$};
\draw (0,-1.25) node {$(i, I)$};
\draw (1,1.25) node {$(i, I)$};
\draw (0.9,-1.25) node {$(j, J)$};
\draw (0.5,0.5) node {$i-j$};
%\draw (1, -0.33) node[circle,fill,inner sep=1pt]{};
%\draw (0, 0.33) node[circle,fill,inner sep=1pt]{};
\end{tikzpicture}\right),\\
%%%%%%%%%%%%%%
&\left(\begin{tikzpicture}[baseline=-0.65ex, thick, scale=0.55]
\draw (1,-1) [->]  -- (-1,1)  node[above]{$(j, J)$};
\draw (-1, -1) to (-0.2, -0.2);
\draw (0.2, 0.2) [->] -- (1, 1) node[above]{$(i, I)$};
\end{tikzpicture}\right)
= \frac{q^{j-i+iJ+jI}}{\{2i\}_q}\cdot
\left(
\begin{tikzpicture}[baseline=-0.65ex, thick, scale=1.1]
\draw (0,-1) [->-] to  (0.5,-0.33);
\draw (0.5, -0.33) [->-] to  (0.5,0.33);
\draw (0.5, 0.33) [->] to  (0,1);
\draw (1,-1) [->-] to  (0.5,-0.33);
\draw (1,1) [<-] to  (0.5,0.33);
\draw (0, -1.25) node {$(i, I)$};
\draw (0.9,-1.25) node {$(j, J)$};
\draw (0,1.25) node {$(j, J)$};
\draw (1,1.25) node {$(i, I)$};
\draw (1,0) node {$i+j$};
%\draw (0.5, -0.33) node[circle,fill,inner sep=1pt]{};
%\draw (0.5, 0.33) node[circle,fill,inner sep=1pt]{};
\end{tikzpicture}\right) 
+  \frac{-q^{i+j+iJ+jI}}{\{2i\}_q}\cdot
\left(\begin{tikzpicture}[baseline=-0.65ex, thick, scale=1.1]
\draw (0,-1) [->-] to (0, 0.33);
\draw (0, 0.33) [->] to (0,1);
\draw (1,-1) [->-] to (1, -0.33);
\draw (1, -0.33) [->] to (1,1);
\draw (0,0.33) [-<-]  to  (1,-0.33);
\draw (0,1.25) node {$(j, J)$};
\draw (0,-1.25) node {$(i, I)$};
\draw (1,1.25) node {$(i, I)$};
\draw (0.9,-1.25) node {$(j, J)$};
\draw (0.5,0.5) node {$j-i$};
%\draw (1, -0.33) node[circle,fill,inner sep=1pt]{};
%\draw (0, 0.33) node[circle,fill,inner sep=1pt]{};
\end{tikzpicture}\right)\\
 & \hspace{3cm}
 = \frac{q^{i-j+iJ+jI}}{\{2j\}_q}\cdot
\left(
\begin{tikzpicture}[baseline=-0.65ex, thick, scale=1.1]
\draw (0,-1) [->-] to  (0.5,-0.33);
\draw (0.5, -0.33) [->-] to  (0.5,0.33);
\draw (0.5, 0.33) [->] to  (0,1);
\draw (1,-1) [->-] to  (0.5,-0.33);
\draw (1,1) [<-] to  (0.5,0.33);
\draw (0, -1.25) node {$(i, I)$};
\draw (0.9,-1.25) node {$(j, J)$};
\draw (0,1.25) node {$(j, J)$};
\draw (1,1.25) node {$(i, I)$};
\draw (1,0) node {$i+j$};
%\draw (0.5, -0.33) node[circle,fill,inner sep=1pt]{};
%\draw (0.5, 0.33) node[circle,fill,inner sep=1pt]{};
\end{tikzpicture}\right)
+\frac{-q^{i+j+iJ+jI}}{\{2j\}_q}\cdot
\left(\begin{tikzpicture}[baseline=-0.65ex, thick, scale=1.1]
\draw (0,-1) [->-] to (0, -0.33);
\draw (0, -0.33) [->] to (0,1);
\draw (1,-1) [->-] to (1, 0.33);
\draw (1, 0.33) [->] to (1,1);
\draw (0,-0.33) [->-]  to  (1,0.33);
\draw (0,1.25) node {$(j, J)$};
\draw (0,-1.25) node {$(i, I)$};
\draw (1,1.25) node {$(i, I)$};
\draw (0.9,-1.25) node {$(j, J)$};
\draw (0.5,0.5) node {$i-j$};
%\draw (1, -0.33) node[circle,fill,inner sep=1pt]{};
%\draw (0, 0.33) node[circle,fill,inner sep=1pt]{};
\end{tikzpicture}\right).
\end{align*}
\end{theo}
\begin{proof}
Relations (i), (ii), (iii) and (iv) were proved in \cite[Sections 5; see also Section 9]{MR2255851}. The other relations can be proved by a straightforward application of the Boltzmann weights in Tables \ref{viro1} and \ref{viro2}. We prove (v), (vii) and the first relation of (viii). It is enough to show that the morphisms between $\mathcal{U}_q(\mathfrak{gl}(1\vert 1))$-modules defined by the local diagrams satisfy the same relations. Proofs for the remaining relations are left to the reader.

\vspace{2mm}

(v) The left-hand diagram has the following local states on which the Boltzmann weights do not vanish.
\begin{align*}
\begin{tikzpicture}[baseline=-0.65ex, thick, scale=1.2]
\draw [dotted](0,-1) [->-] to (0, -0.43);
\draw [dotted](0,0.43) [->-] to (0, -0.43);
\draw [dotted](0, 0.43) [->] to (0,1);
\draw [dotted](1,-1) [<-] to (1, -0.1);
\draw [dotted](1, -0.1) [->-] to (1,0.43);
\draw  [dotted] (1,0.43) [-<-] to (1,1);
\draw [dotted](1, 0.33) [->-] to (0,0.63);
\draw [dotted](0,-0.43) [->-] to (1,-0.1);
%\draw (0,1.25) node {$j$};
%\draw (0,-1.25) node {$i$};
%\draw (1,1.25) node {$i$};
%\draw (0.9,-1.25) node {$j$};
%\draw (0.5,0.7) node {$i+j$};
%\draw (1, 0.33) node[circle,fill,inner sep=1pt]{};
%\draw (0, -0.33) node[circle,fill,inner sep=1pt]{};
\end{tikzpicture} \quad
\begin{tikzpicture}[baseline=-0.65ex, thick, scale=1.2]
\draw [dotted](0,-1) [->-] to (0, -0.43);
\draw (0,0.63) [->-] to (0, -0.43);
\draw [dotted] (0, 0.43) [->] to (0,1);
\draw [dotted](1,-1) [<-] to (1, -0.1);
\draw (1, -0.1) [->-] to (1,0.33);
\draw  [dotted] (1,0.33) [-<-] to (1,1);
\draw (1, 0.33) [->-] to (0,0.63);
\draw (0,-0.43) [->-] to (1,-0.1);
\end{tikzpicture} \quad
\begin{tikzpicture}[baseline=-0.65ex, thick, scale=1.2]
\draw [dotted] (0,-1) [->-] to (0, -0.43);
\draw [dotted] (0,0.63) [->-] to (0, -0.43);
\draw  (0, 0.63) [->] to (0,1);
\draw [dotted] (1,-1) [<-] to (1, -0.1);
\draw [dotted] (1, -0.1) [->-] to (1,0.33);
\draw  (1,0.33) [-<-] to (1,1);
\draw (1, 0.33) [->-] to (0,0.63);
\draw [dotted] (0,-0.43) [->-] to (1,-0.1);
\end{tikzpicture} \quad
\begin{tikzpicture}[baseline=-0.65ex, thick, scale=1.2]
\draw (0,-1) [->-] to (0, -0.43);
\draw [dotted] (0,0.63) [->-] to (0, -0.43);
\draw [dotted] (0, 0.43) [->] to (0,1);
\draw (1,-1) [<-] to (1, -0.1);
\draw [dotted] (1, -0.1) [->-] to (1,0.33);
\draw [dotted] (1,0.33) [-<-] to (1,1);
\draw [dotted] (1, 0.33) [->-] to (0,0.63);
\draw  (0,-0.43) [->-] to (1,-0.1);
\end{tikzpicture} \quad
\begin{tikzpicture}[baseline=-0.65ex, thick, scale=1.2]
\draw (0,-1) [->-] to (0, -0.43);
\draw [dotted] (0,0.63) [->-] to (0, -0.43);
\draw  (0, 0.63) [->] to (0,1);
\draw (1,-1) [<-] to (1, -0.1);
\draw [dotted] (1, -0.1) [->-] to (1,0.33);
\draw  (1,0.33) [-<-] to (1,1);
\draw  (1, 0.33) [->-] to (0,0.63);
\draw  (0,-0.43) [->-] to (1,-0.1);
\end{tikzpicture} \quad
\begin{tikzpicture}[baseline=-0.65ex, thick, scale=1.2]
\draw (0,-1) [->-] to (0, -0.43);
\draw [dotted] (0,0.63) [->-] to (0, -0.43);
\draw  (0, 0.63) [->] to (0,1);
\draw [dotted] (1,-1) [<-] to (1, -0.1);
\draw (1, -0.1) [->-] to (1,0.33);
\draw  [dotted] (1,0.33) [-<-] to (1,1);
\draw (1, 0.33) [->-] to (0,0.63);
\draw (0,-0.43) [->-] to (1,-0.1);
\end{tikzpicture} \quad
\begin{tikzpicture}[baseline=-0.65ex, thick, scale=1.2]
\draw [dotted](0,-1) [->-] to (0, -0.43);
\draw (0,0.63) [->-] to (0, -0.43);
\draw [dotted] (0, 0.43) [->] to (0,1);
\draw (1,-1) [<-] to (1, -0.1);
\draw [dotted] (1, -0.1) [->-] to (1,0.33);
\draw  (1,0.33) [-<-] to (1,1);
\draw (1, 0.33) [->-] to (0,0.63);
\draw (0,-0.43) [->-] to (1,-0.1);
\end{tikzpicture}
\end{align*}
Therefore the morphism defined by the left-hand diagram is as below. For simplicity we omit the subindex $q$ from the expression $\{k\}_q$.
\begin{eqnarray*}
e_0 \otimes e_0 &\mapsto & q^{-2i}e_0\otimes e_0\otimes e_0 + e_1 \otimes e_1\otimes e_0\\
&\mapsto & q^{-2i}(\{2i+2j\}q^{2i}) e_0 \otimes e_0 + \{2j\}e_1 \otimes e_1\\
& \mapsto & q^{-2i}(\{2i+2j\}q^{2i})[ e_0 \otimes e_0\otimes e_0 + (-q^{2(i+j)}) e_0 \otimes e_1\otimes e_1]
+ \{2j\} e_1 \otimes e_1\otimes e_0 \\
& \mapsto & \{2i+2j\}[\{2i+2j\} e_0 \otimes e_0 -q^{2(i+j)}\{2j\} e_1 \otimes e_1] -\{2j\}\{2i\}q^{-2(i+j)} e_0\otimes e_0\\
&=& (\{2i+2j\}^2-\{2i\}\{2j\}q^{-2(i+j)}) e_0\otimes e_0 -\{2i+2j\}\{2j\} q^{2(i+j)} e_1\otimes e_1;\\
e_0 \otimes e_1 & \mapsto & e_1 \otimes e_1\otimes e_1\\
& \mapsto & \{2i\} e_1 \otimes e_0\\
& \mapsto & \{2i\}(-q^{2(i+j)})e_1 \otimes e_1\otimes e_1\\
& \mapsto & \{2i\}(-q^{2(i+j)})(-\{2i\}q^{-2(i+j)})e_0\otimes e_1=\{2i\}^2 e_0 \otimes e_1;\\
e_1\otimes e_0 & \mapsto & e_0 \otimes e_1\otimes e_0\\
& \mapsto & \{2j\} e_0\otimes e_1\\
& \mapsto & \{2j\} e_0\otimes e_1\otimes e_0\\
& \mapsto &  \{2j\}^2  e_1\otimes e_0;\\
e_1\otimes e_1 & \mapsto & e_0\otimes e_1 \otimes e_1\\
& \mapsto & \{2i\} e_0 \otimes e_0\\
& \mapsto & \{2i\}[e_0 \otimes e_0 \otimes e_0 + (-q^{2(i+j)})e_0\otimes e_1 \otimes e_1]\\
& \mapsto & \{2i\} [\{2i+2j\} e_0 \otimes e_0 + (-q^{2(i+j)}) \{2j\}e_1 \otimes e_1]\\
& = & \{2i\} \{2i+2j\} e_0 \otimes e_0 -\{2i\}\{2j\}q^{2(i+j)} e_1 \otimes e_1.
\end{eqnarray*}
The matrix for the morphism under the basis $\{e_0 \otimes e_0, e_0 \otimes e_1, e_1 \otimes e_0, e_1 \otimes e_1\}$ is
$$
A= \left(
    \begin{array}{cccc}
      \{2i+2j\}^2-\{2i\}\{2j\}q^{-2(i+j)} & 0 & 0 & \{2i\} \{2i+2j\}\\
      0 & \{2i\}^2 & 0 & 0 \\
      0 & 0 & \{2j\}^2 & 0 \\
      -\{2i+2j\}\{2j\} q^{2(i+j)} & 0 & 0 & -\{2i\}\{2j\}q^{2(i+j)}
    \end{array}
  \right).
$$

The two local diagrams on the right-hand side have the following local states on which the Boltzmann weights do not vanish:
\begin{align*}
\begin{tikzpicture}[baseline=-0.65ex, thick, scale=1.2]
\draw [dotted] (0,-1) [->-] to  (0.5,-0.33);
\draw [dotted] (0.5, -0.33) [-<-] to  (0.5,0.33);
\draw [dotted] (0.5, 0.33) [->] to  (0,1);
\draw [dotted] (1,-1) [<-] to  (0.5,-0.33);
\draw [dotted] (1,1) [->-] to  (0.5,0.33);
%\draw (0, -1.25) node {$i$};
%\draw (0.9,-1.25) node {$j$};
%\draw (0,1.25) node {$j$};
%\draw (1,1.25) node {$i$};
%\draw (1,0) node {$i+j$};
%\draw (0.5, -0.33) node[circle,fill,inner sep=1pt]{};
%\draw (0.5, 0.33) node[circle,fill,inner sep=1pt]{};
\end{tikzpicture} \quad
\begin{tikzpicture}[baseline=-0.65ex, thick, scale=1.2]
\draw [dotted] (0,-1) [->-] to  (0.5,-0.33);
\draw [dotted] (0.5, -0.33) [-<-] to  (0.5,0.33);
\draw (0.5, 0.33) [->] to  (0,1);
\draw [dotted] (1,-1) [<-] to  (0.5,-0.33);
\draw (1,1) [->-] to  (0.5,0.33);
\end{tikzpicture} \quad 
\begin{tikzpicture}[baseline=-0.65ex, thick, scale=1.2]
\draw  (0,-1) [->-] to  (0.5,-0.33);
\draw [dotted] (0.5, -0.33) [-<-] to  (0.5,0.33);
\draw [dotted] (0.5, 0.33) [->] to  (0,1);
\draw  (1,-1) [<-] to  (0.5,-0.33);
\draw [dotted] (1,1) [->-] to  (0.5,0.33);
\end{tikzpicture} \quad
\begin{tikzpicture}[baseline=-0.65ex, thick, scale=1.2]
\draw  (0,-1) [->-] to  (0.5,-0.33);
\draw [dotted] (0.5, -0.33) [-<-] to  (0.5,0.33);
\draw (0.5, 0.33) [->] to  (0,1);
\draw  (1,-1) [<-] to  (0.5,-0.33);
\draw (1,1) [->-] to  (0.5,0.33);
\end{tikzpicture} \quad
\begin{tikzpicture}[baseline=-0.65ex, thick, scale=1.2]
\draw [dotted] (0,-1) [->-] to  (0.5,-0.33);
\draw  (0.5, -0.33) [-<-] to  (0.5,0.33);
\draw [dotted] (0.5, 0.33) [->] to  (0,1);
\draw   (1,-1) [<-] to  (0.5,-0.33);
\draw  (1,1) [->-] to  (0.5,0.33);
\end{tikzpicture},
\end{align*} and
\begin{align*} 
\begin{tikzpicture}[baseline=-0.65ex, thick, scale=1.2]
\draw (0,-1) [->] to (0,0.5);
\draw (0.5,-1) [<-] to (0.5,0.5);
%\draw (-0.25,0) node {$j$};
%\draw (1.25, 0) node {$i$};
\end{tikzpicture} \quad \quad
\begin{tikzpicture}[baseline=-0.65ex, thick, scale=1.2]
\draw [dotted] (0,-1) [->] to (0,0.5);
\draw [dotted] (0.5,-1) [<-] to (0.5,0.5);
\end{tikzpicture}  \quad \quad
\begin{tikzpicture}[baseline=-0.65ex, thick, scale=1.2]
\draw  (0,-1) [->] to (0,0.5);
\draw [dotted] (0.5,-1) [<-] to (0.5,0.5);
\end{tikzpicture}  \quad \quad
\begin{tikzpicture}[baseline=-0.65ex, thick, scale=1.2]
\draw [dotted] (0,-1) [->] to (0,0.5);
\draw  (0.5,-1) [<-] to (0.5,0.5);
\end{tikzpicture}.
\end{align*} 
The matrices for the morphisms are 
$$
  B = \left(
    \begin{array}{cccc}
      \{2i\}q^{2j} & 0 & 0 & \{2i\}\\
      0 & \{2i-2j\} & 0 & 0 \\
      0 & 0 & 0 & 0 \\
     -\{2j\}q^{2(i+j)} & 0 & 0 & -\{2j\}q^{2i}
    \end{array}
  \right)
$$ and the identity matrix $I$ respectively. We can check the equality $A=\{2i+2j\} \cdot B +\{2j\}^2 \cdot I$ by applying the identities in Lemma \ref{identity}. Consider the $(1, 1)$-entry, we have
\begin{eqnarray*}
&&\{2i+2j\}^2-\{2i\}\{2j\}q^{-2(i+j)}\\
&\overset{\eqref{eq1}}{=}&\{2i+2j\}^2-(\{2i+2j\}-q^{2i}\{2j\})(\{2i+2j\}-q^{2j}\{2i\})\\
&=& \{2i+2j\}\{2i\}q^{2j}+\{2i+2j\}\{2j\}q^{2i}-\{2i\}\{2j\}q^{2i+2j}\\
&=& \{2i+2j\}\{2i\}q^{2j}+\{2j\}q^{2i}(\{2i+2j\}-\{2i\}q^{2j})\\
&\overset{\eqref{eq1}}{=}& \{2i+2j\}\{2i\}q^{2j}+\{2j\}q^{2i}\{2j\}q^{-2i}=\{2i+2j\}\{2i\}q^{2j}+\{2j\}^2.
\end{eqnarray*}
The equality on the $(2, 2)$-entry needs 
$$\{2i\}^2=\{2i+2j\}\{2i-2j\}+\{2j\}^2$$
which follows from \eqref{eq2} of Lemma \ref{identity}. For the $(4, 4)$-entry we have
\begin{eqnarray*}
-\{2i\}\{2j\}q^{2(i+j)}&\overset{\eqref{eq1}}{=}&-\{2j\}q^{2i}(\{2i+2j\}-\{2j\}q^{-2i})\\
&=& -\{2i+2j\}\{2j\}q^{2i}+\{2j\}^2.
\end{eqnarray*}
The equalities on the other entries can be verified by direct comparison.

\vspace{2mm}

(vii) The left-hand local diagram has the following local states with non-vanishing Boltzmann weights.
\begin{align*}
\begin{tikzpicture}[baseline=-0.65ex, thick, scale=1.2]
\draw [dotted](0,-1) [->-] to (0, -0.43);
\draw [dotted] (0,0.63) [-<-] to (0, -0.43);
\draw [dotted] (0, 0.63) [->] to (0,1);
\draw [dotted] (-1,-1) [->-] to (-1, -0.1);
\draw [dotted] (-1, -0.1) [->-] to (-1,0.33);
\draw [dotted] (-1,0.33) [->] to (-1,1);
\draw [dotted] (-1, 0.33) [->-] to (0,0.63);
\draw [dotted] (0,-0.43) [->-] to (-1,-0.1);
\end{tikzpicture}\quad
\begin{tikzpicture}[baseline=-0.65ex, thick, scale=1.2]
\draw (0,-1) [->-] to (0, -0.43);
\draw  (0,0.63) [-<-] to (0, -0.43);
\draw  (0, 0.63) [->] to (0,1);
\draw [dotted] (-1,-1) [->-] to (-1, -0.1);
\draw [dotted] (-1, -0.1) [->-] to (-1,0.33);
\draw [dotted] (-1,0.33) [->] to (-1,1);
\draw [dotted] (-1, 0.33) [->-] to (0,0.63);
\draw [dotted] (0,-0.43) [->-] to (-1,-0.1);
\end{tikzpicture}\quad
\begin{tikzpicture}[baseline=-0.65ex, thick, scale=1.2]
\draw (0,-1) [->-] to (0, -0.43);
\draw [dotted] (0,0.63) [-<-] to (0, -0.43);
\draw  (0, 0.63) [->] to (0,1);
\draw [dotted] (-1,-1) [->-] to (-1, -0.1);
\draw  (-1, -0.1) [->-] to (-1,0.33);
\draw [dotted] (-1,0.33) [->] to (-1,1);
\draw  (-1, 0.33) [->-] to (0,0.63);
\draw  (0,-0.43) [->-] to (-1,-0.1);
\end{tikzpicture}\quad
\begin{tikzpicture}[baseline=-0.65ex, thick, scale=1.2]
\draw (0,-1) [->-] to (0, -0.43);
\draw [dotted] (0,0.63) [-<-] to (0, -0.43);
\draw [dotted] (0, 0.63) [->] to (0,1);
\draw [dotted] (-1,-1) [->-] to (-1, -0.1);
\draw  (-1, -0.1) [->-] to (-1,0.33);
\draw  (-1,0.33) [->] to (-1,1);
\draw  [dotted] (-1, 0.33) [->-] to (0,0.63);
\draw  (0,-0.43) [->-] to (-1,-0.1);
\end{tikzpicture}\quad
\begin{tikzpicture}[baseline=-0.65ex, thick, scale=1.2]
\draw [dotted] (0,-1) [->-] to (0, -0.43);
\draw [dotted] (0,0.63) [-<-] to (0, -0.43);
\draw [dotted] (0, 0.63) [->] to (0,1);
\draw  (-1,-1) [->-] to (-1, -0.1);
\draw  (-1, -0.1) [->-] to (-1,0.33);
\draw  (-1,0.33) [->] to (-1,1);
\draw  [dotted] (-1, 0.33) [->-] to (0,0.63);
\draw [dotted] (0,-0.43) [->-] to (-1,-0.1);
\end{tikzpicture}\quad
\begin{tikzpicture}[baseline=-0.65ex, thick, scale=1.2]
\draw [dotted] (0,-1) [->-] to (0, -0.43);
\draw [dotted] (0,0.63) [-<-] to (0, -0.43);
\draw  (0, 0.63) [->] to (0,1);
\draw  (-1,-1) [->-] to (-1, -0.1);
\draw  (-1, -0.1) [->-] to (-1,0.33);
\draw [dotted] (-1,0.33) [->] to (-1,1);
\draw   (-1, 0.33) [->-] to (0,0.63);
\draw [dotted] (0,-0.43) [->-] to (-1,-0.1);
\end{tikzpicture}\quad
\begin{tikzpicture}[baseline=-0.65ex, thick, scale=1.2]
\draw (0,-1) [->-] to (0, -0.43);
\draw  (0,0.63) [-<-] to (0, -0.43);
\draw  (0, 0.63) [->] to (0,1);
\draw  (-1,-1) [->-] to (-1, -0.1);
\draw  (-1, -0.1) [->-] to (-1,0.33);
\draw  (-1,0.33) [->] to (-1,1);
\draw  [dotted] (-1, 0.33) [->-] to (0,0.63);
\draw [dotted] (0,-0.43) [->-] to (-1,-0.1);
\end{tikzpicture}
\end{align*}
Therefore the morphism defined by the left-hand diagram is as below.
\begin{eqnarray*}
e_0 \otimes e_0 &\mapsto & \{2j\} e_0 \otimes e_0 \otimes e_0\\
&\mapsto & \{2j\} e_0 \otimes e_0 \\
&\mapsto & \{2j\}\{2i+2k\} e_0 \otimes e_0 \otimes e_0\\
&\mapsto & \{2j\}\{2i+2k\} e_0 \otimes e_0;\\
e_0 \otimes e_1 &\mapsto & \{2j-2k\} e_0\otimes e_0 \otimes e_1 +\{2k\}q^{-2(j-k)}e_0 \otimes e_1 \otimes e_0\\
&\mapsto & \{2j-2k\} e_0 \otimes e_1 +\{2k\}q^{-2(j-k)}q^{2i} e_1 \otimes e_0\\
&\mapsto & \{2j-2k\}\{2i+2k\} e_0 \otimes e_0 \otimes e_1\\
&& +\{2k\}q^{-2(j-k)}q^{2i} [ \{2l\}e_0 \otimes e_1 \otimes e_0 + \{2i+2k-2l\}q^{-2l}e_1 \otimes e_0 \otimes e_0]\\
&\mapsto & \{2j-2k\}\{2i+2k\}q^{2l} e_0 \otimes e_1\\
&& +\{2k\}q^{-2(j-k)+2i} [ \{2l\}e_0 \otimes e_1 + \{2i+2k-2l\}q^{-2l}e_1 \otimes e_0]\\
&=& (\{2j-2k\}\{2i+2k\}q^{2l} +\{2k\}\{2l\}q^{2(i-j+k)}) e_0 \otimes e_1\\
&& +  \{2i+2k-2l\}\{2k\} q^{2(i-j+k-l)}e_1 \otimes e_0;\\
e_1\otimes e_0 &\mapsto & \{2j\} e_1 \otimes e_0 \otimes e_0\\
&\mapsto & \{2j\} e_1  \otimes e_0\\
&\mapsto & \{2j\} [\{2i+2k-2l\}q^{-2l} e_1 \otimes e_0 \otimes e_0+ \{2l\} e_0 \otimes e_1 \otimes e_0]\\
&\mapsto & \{2j\} \{2i+2k-2l\}q^{-2l} e_1 \otimes e_0 + \{2j\}\{2l\} e_0 \otimes e_1;\\
e_1\otimes e_1 &\mapsto & \{2j-2k\} e_1 \otimes e_0 \otimes e_1\\
&\mapsto & \{2j-2k\} e_1 \otimes e_1\\
&\mapsto & \{2j-2k\}\{2i+2k-2l\}q^{-2l} e_1 \otimes e_0 \otimes e_1\\
&\mapsto & \{2j-2k\}\{2i+2k-2l\}q^{-2l} q^{2l}e_1  \otimes e_1=\{2j-2k\}\{2i+2k-2l\}e_1  \otimes e_1.
\end{eqnarray*}
The matrix for the morphism is 
$$
A= \left(
    \begin{array}{cccc}
      \{2j\}\{2i+2k\} & 0 & 0 & 0\\
      0 & a_{22} & \{2j\}\{2l\} & 0 \\
      0 & a_{32}  & \{2j\} \{2i+2k-2l\}q^{-2l} & 0 \\
      0 & 0 & 0 & \{2j-2k\}\{2i+2k-2l\}
    \end{array}
  \right),
$$ where $a_{22}=\{2j-2k\}\{2i+2k\}q^{2l} +\{2k\}\{2l\}q^{2(i-j+k)}$ and $a_{32}=\{2i+2k-2l\}\{2k\} q^{2(i-j+k-l)}$.

The two diagrams on the right-hand side have the following local states with non-vanishing Boltzmann weights:
\begin{align*}
\begin{tikzpicture}[baseline=-0.65ex, thick, scale=1.2]
\draw [dotted] (0,-1) [->-] to  (0.5,-0.33);
\draw [dotted] (0.5, -0.33) [->-] to  (0.5,0.33);
\draw [dotted] (0.5, 0.33) [->] to  (0,1);
\draw [dotted] (1,-1) [->-] to  (0.5,-0.33);
\draw [dotted] (1,1) [<-] to  (0.5,0.33);
\end{tikzpicture} \quad
\begin{tikzpicture}[baseline=-0.65ex, thick, scale=1.2]
\draw [dotted] (0,-1) [->-] to  (0.5,-0.33);
\draw  (0.5, -0.33) [->-] to  (0.5,0.33);
\draw [dotted] (0.5, 0.33) [->] to  (0,1);
\draw  (1,-1) [->-] to  (0.5,-0.33);
\draw  (1,1) [<-] to  (0.5,0.33);
\end{tikzpicture} \quad
\begin{tikzpicture}[baseline=-0.65ex, thick, scale=1.2]
\draw [dotted] (0,-1) [->-] to  (0.5,-0.33);
\draw  (0.5, -0.33) [->-] to  (0.5,0.33);
\draw  (0.5, 0.33) [->] to  (0,1);
\draw  (1,-1) [->-] to  (0.5,-0.33);
\draw  [dotted] (1,1) [<-] to  (0.5,0.33);
\end{tikzpicture} \quad
\begin{tikzpicture}[baseline=-0.65ex, thick, scale=1.2]
\draw  (0,-1) [->-] to  (0.5,-0.33);
\draw  (0.5, -0.33) [->-] to  (0.5,0.33);
\draw  (0.5, 0.33) [->] to  (0,1);
\draw  [dotted] (1,-1) [->-] to  (0.5,-0.33);
\draw  [dotted] (1,1) [<-] to  (0.5,0.33);
\end{tikzpicture} \quad
\begin{tikzpicture}[baseline=-0.65ex, thick, scale=1.2]
\draw  (0,-1) [->-] to  (0.5,-0.33);
\draw  (0.5, -0.33) [->-] to  (0.5,0.33);
\draw   [dotted] (0.5, 0.33) [->] to  (0,1);
\draw  [dotted] (1,-1) [->-] to  (0.5,-0.33);
\draw  (1,1) [<-] to  (0.5,0.33);
\end{tikzpicture} \quad
\end{align*} and
\begin{align*}
\begin{tikzpicture}[baseline=-0.65ex, thick, scale=1.2]
\draw [dotted] (0,-1) [->-] to (0, 0.33);
\draw [dotted] (0, 0.33) [->] to (0,1);
\draw [dotted] (1,-1) [->-] to (1, -0.33);
\draw [dotted] (1, -0.33) [->] to (1,1);
\draw [dotted] (0,0.33) [-<-]  to  (1,-0.33);
\end{tikzpicture}\quad
\begin{tikzpicture}[baseline=-0.65ex, thick, scale=1.2]
\draw [dotted] (0,-1) [->-] to (0, 0.33);
\draw [dotted] (0, 0.33) [->] to (0,1);
\draw (1,-1) [->-] to (1, -0.33);
\draw (1, -0.33) [->] to (1,1);
\draw [dotted] (0,0.33) [-<-]  to  (1,-0.33);
\end{tikzpicture}\quad
\begin{tikzpicture}[baseline=-0.65ex, thick, scale=1.2]
\draw [dotted] (0,-1) [->-] to (0, 0.33);
\draw  (0, 0.33) [->] to (0,1);
\draw (1,-1) [->-] to (1, -0.33);
\draw [dotted] (1, -0.33) [->] to (1,1);
\draw  (0,0.33) [-<-]  to  (1,-0.33);
\end{tikzpicture}\quad
\begin{tikzpicture}[baseline=-0.65ex, thick, scale=1.2]
\draw  (0,-1) [->-] to (0, 0.33);
\draw  (0, 0.33) [->] to (0,1);
\draw [dotted] (1,-1) [->-] to (1, -0.33);
\draw [dotted] (1, -0.33) [->] to (1,1);
\draw  [dotted] (0,0.33) [-<-]  to  (1,-0.33);
\end{tikzpicture}\quad
\begin{tikzpicture}[baseline=-0.65ex, thick, scale=1.2]
\draw  (0,-1) [->-] to (0, 0.33);
\draw  (0, 0.33) [->] to (0,1);
\draw  (1,-1) [->-] to (1, -0.33);
\draw  (1, -0.33) [->] to (1,1);
\draw  [dotted] (0,0.33) [-<-]  to  (1,-0.33);
\end{tikzpicture}.
\end{align*}
The matrices for the morphisms are 
$$
B = \left(
    \begin{array}{cccc}
      \{2i+2j\} & 0 & 0 & 0\\
      0 & \{2j+2l-2k\}q^{2i} & \{2j+2l-2k\} & 0 \\
      0 & \{2i+2k-2l\}q^{-2(j+l-k-i)}  & \{2i+2k-2l\}q^{-2(j+l-k)} & 0 \\
      0 & 0 & 0 & 0
    \end{array}
  \right)
$$ and 
$$
C = \left(
    \begin{array}{cccc}
      \{2j\} & 0 & 0 & 0\\
      0 & \{2j+2l-2k\} & 0& 0 \\
      0 & \{2k-2l\}q^{-2(j+l-k-i)} & \{2j\} & 0 \\
      0 & 0 & 0 & \{2j+2l-2k\}
    \end{array}
  \right).
$$
We claim that $\displaystyle A=\frac{\{2j\}\{2l\}}{\{2j+2l-2k\}}\cdot B + \frac{\{2i+2k-2l\}\{2j-2k\}}{\{2j+2l-2k\}} \cdot C$.
For the $(1, 1)$-entry, we have
\begin{eqnarray*}
\{2j\}\{2i+2k\}&=&\{2j\}\frac{\{2i+2k\}\{2j+2l-2k\}}{\{2j+2l-2k\}}\\
&\overset{\eqref{eq3}}{=}&\{2j\}\frac{\{2i+2j\}\{2l\}+\{2i+2k-2l\}\{2j-2k\}}{\{2j+2l-2k\}}.\\
\end{eqnarray*}
In the second equality, we apply \eqref{eq3} by regarding $a=2j+2l-2k$, $b=2i+2k-2l$ and $c=2l$.
For the $(2,2)$-entry, we have
$$
\{2j-2k\}\{2i+2k\}q^{2l}\overset{\eqref{eq1}}{=}\{2j-2k\}(\{2i+2k-2l\}+\{2l\}q^{2i+2k})\\
$$ by regarding $a=-2l$ and $b=2i+2k$ in \eqref{eq1}, and
$$
\{2k\}\{2l\}q^{2(i-j+k)}\overset{\eqref{eq1}}{=}\{2l\}q^{2i}(\{2j\}-\{2j-2k\}q^{2k})
$$ by regarding $a=2j-2k$ and $b=2k$ in \eqref{eq1}. Therefore 
$$a_{22}=\{2j\}\{2l\}q^{2i}+\{2i+2k-2l\}\{2j-2k\}.$$
For the $(3,2)$-entry, by comparing both sides, it is enough to verify that
$$
\{2k\}=\frac{\{2j\}\{2l\}+\{2k-2l\}\{2j-2k\}}{\{2j+2l-2k\}},
$$
which follows from \eqref{eq3} of Lemma \ref{identity} by regarding $a=2k$, $b=2j-2k$ and $c=2l$. For the $(3, 3)$-entry, by comparing both sides, it is enough to verify that
$$
q^{-2l}=\frac{\{2l\}q^{-2(j+l-k)}+\{2j-2k\}}{\{2j+2l-2k\}},
$$
which follows from \eqref{eq1} of Lemma \ref{identity} by regarding $a=2j-2k$ and $b=2l$. The equality on the other entries can be checked by direct comparison.

\vspace{2mm}

(viii) We prove the first equality of (viii). From Table 1, we can see that the matrix for the diagram on the left-hand side is
$$ 
A = \left(
    \begin{array}{cccc}
      q^{i+j-iJ-jI} & 0 & 0 & 0\\
      0 & \displaystyle \frac{(q^{4i}-1)}{q^{i+j+iJ+jI}} & q^{i-j-iJ-jI}  & 0 \\
      0 & q^{j-i-iJ-jI}  & 0 & 0 \\
      0 & 0 & 0 & -q^{-i-j-iJ-jI}
    \end{array}
  \right).
$$ 
For the diagrams on the right-hand side, the morphisms for them have been calculated in the proof of (vii). Let $k-l=j-i$. The matrices of the morphisms are respectively
$$
B = \left(
    \begin{array}{cccc}
      \{2i+2j\} & 0 & 0 & 0\\
      0 & \{2i\}q^{2i} & \{2i\} & 0 \\
      0 & \{2j\}  & \{2j\}q^{-2i} & 0 \\
      0 & 0 & 0 & 0
    \end{array}
  \right)
$$ and 
$$
C = \left(
    \begin{array}{cccc}
      \{2j\} & 0 & 0 & 0\\
      0 & \{2i\} & 0 & 0 \\
      0 & \{2j-2i\}  & \{2j\} & 0 \\
      0 & 0 & 0 & \{2i\}
    \end{array}
  \right).
$$
We see that $\displaystyle A=\frac{q^{i-j-iJ-jI}}{\{2i\}}B+\frac{-q^{-i-j-iJ-jI}}{\{2i\}}C$.
\end{proof}

%%%%%%%%%%%%%%%%%%%%%%%
%%%%%%%%%%%%%%%%%%%%%%%

\section{Relations between $\Delta_{(G, c)}(t)$ and $\underline{\Delta}(G, c)$}
For a graph $G$, consider a coloring $c$ of it for which the weight on each edge is one. It is easy to see that the admissibility condition for the weights is satisfied. In this case, we simply suppress the weights and the coloring becomes a map $c: E\to \mathbb{Z}\backslash \{0\}$ which sends each edge to its multiplicity. The admissibility conditions simply become one condition $\sum_{i=1}^{3} \epsilon_i j_i=0$, which coincides with the definition of balanced coloring in \cite{alex}. We recall its definition here.

\begin{defn}
\rm
For a graph $G$, let $E$ be the set of its edges. 
A map $c: E\to \mathbb{Z}$ is called a {\it balanced coloring} if for each vertex $v$ of $G$, we have
$$\sum_{e: \text{ pointing into } v} c(e)=\sum_{e: \text{ pointing out of } v} c(e).$$ If further each $c(e)>0$, we call $c$ a {\it positive coloring}.
\end{defn}

For a positive coloring $c$, in \cite{alex} we defined a single-variable polynomial $\Delta_{(G, c)}(t)$ and studied its MOY-type relations.

We briefly review the definition of $\Delta_{(G, c)}(t)$. For more details and properties of it, please refer to \cite{alex}. Let $X$ be the complement of $G$ in $S^3$, which is obtained by removing from $S^3$ the interior of a tubular neighborhood of $G$ in $S^3$. The balanced coloring $c$ of $G$ naturally defines a homomorphism 
\begin{eqnarray*}
\phi_{c}: \pi_1(X, x_0) &\to& \mathbb{Z}\langle t \rangle\\
\text{oriented meridian of $e$} &\mapsto& t^{c(e)},
\end{eqnarray*}
where $e\in E$ and $\mathbb{Z}\langle t \rangle=\{t^k \mid k\in \mathbb{Z}\}$ is the abelian group generated by $t$. Then $\ker (\phi_{c})$ corresponds to a regular covering space of $X$, which we call $p: \widetilde X \to X$.
Consider the $\mathbb{Z} [t^{-1}, t]$-module $H_1 (\widetilde X, p^{-1} (\partial_{\mathrm{in}} (X)))$, where $\partial_{\mathrm{in}} (X):=\bigcup_{v\in V}\partial_{\mathrm{in}} (v)\subset \partial (X)$ and $\partial_{\mathrm{in}} (v)$ is a subsurface around vertex $v$ 
bounded by meridians of the edges pointing toward $v$ and the ``meridian" around $v$, as shown in Fig. \ref{e10}.
The polynomial $\Delta_{(G, c)}(t)$ is the $0$-th characteristic polynomial of a presentation matrix of the module (See \cite[Section 7.2]{MR1417494} for the definition of characteristic polynomial). Modulo $\pm \mathbb{Z}\langle t \rangle$, it becomes a topological invariant of $G$.

\begin{figure}
	\centering
		\includegraphics[width=0.6\textwidth]{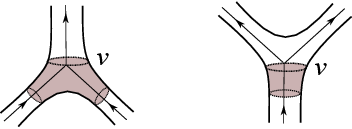}
	\caption{The shadow region indicates $\partial_{\mathrm{in}} (v)$ around a vertex.}
	\label{e10}
\end{figure}

A multi-variable version of $\Delta_{(G, c)}(t)$ was first studied in \cite{bao}, where we showed that it is the Euler characteristic of the Heegaard Floer homology of $G$ studied there. This polynomial was further studied intensively in \cite{alex}, where we provided a normalized state sum formula for it, proved its topological invariance in combinatorial way and studied its MOY-type relations.

To state the relation between $\Delta_{(G, c)}(t)$ and $\underline{\Delta}(G, c)$, we divide the vertices of $G$ into odd and even types.
\begin{align*}
\begin{tikzpicture}[baseline=-0.65ex, thick, scale=1.2]
%\draw  (0,-1) [->-] to (0, 0.33);
\draw  (0, 0) [->-] to (0,0.5);
%\draw  (1,-1) [->-] to (1, -0.33);
%\draw  (1, -0.33) [->] to (1,1);
\draw   (0,0.5) [->]  to  (0.66,1);
\draw   (0,0.5) [->]  to  (-0.66,1);
\draw (0,-0.5) node {odd type};
\end{tikzpicture} \hspace{2cm}
\begin{tikzpicture}[baseline=-0.65ex, thick, scale=1.2]
%\draw  (0,-1) [->-] to (0, 0.33);
\draw  (0, 0.5) [->] to (0,1);
%\draw  (1,-1) [->-] to (1, -0.33);
%\draw  (1, -0.33) [->] to (1,1);
\draw   (0,0.5) [-<-]  to  (0.66,0);
\draw   (0,0.5) [-<-]  to  (-0.66,0);
\draw (0,-0.5) node {even type};
\end{tikzpicture}
\end{align*}
Define {\it the color of a vertex} $v$, which we call $c(v)$, to be the sum of the colors of the edges pointing out of $v$, which by admissibility conditions equals the sum of colors of the edges pointing into $v$. Now we are ready to prove the main result.

\begin{theo}
\label{maintheorem}
For a graph $G$, let $c$ be a coloring whose multiplicities are positive and weights are one. We have 
$$
\Delta_{(G, c)}(q^{-4}) =\frac{\prod_{\text{$v$: even type}} \{2c(v)\}_q }{(q^{2}-q^{-2})^{\vert V \vert-1}}\underline{\Delta}(G, c),
$$ where $\vert V \vert$ is the number of vertices in $G$.
\end{theo}
\begin{proof}
We proved in \cite[Theorem 4.3]{alex} that the MOY-type relations (i)--(x) in \cite[Theorem 4.1]{alex} determine $\Delta_{(G, c)}(t)$. Putting $t=q^{-4}$, we just need to show that the right-hand side of Theorem \ref{maintheorem} satisfies the same relations, by applying Theorem \ref{moy}. Precisely, we get the following table by comparing the relations in \cite[Theorem 4.1]{alex} (the first row of the table) and those in Theorem \ref{moy} (the second row).

\begin{table}[h]
\renewcommand{\arraystretch}{1.2}
 \begin{tabular}{|c|c|c|c|c|c|c|c|} \hline
\multicolumn{1}{|c|}{} & \multicolumn{7}{|c|}{equivalent pairs of relations}  \\
\hline
$\Delta_{(G, c)}(q^{-4})$ & (i, ii, iii) & (iv) & (v, vi)& (vii) & (viii) & (ix) & (x)\\ 
\hline 
  \vspace{-2mm}   & &&&&&& \\
$ \frac{\prod_{\text{$v$: even type}} \{2c(v)\}_q }{(q^{2}-q^{-2})^{\vert V \vert-1}}\underline{\Delta}(G, c)$ & (i, ii, iii) &(viii)& (iv)& (v)&(vi) & (vii)& (0)  \\[7pt]
\hline
 \end{tabular}
\end{table}
\end{proof}

%%%%%%%%% 
%%%%%%%%% new section

\section{A concrete correspondence when $G$ is a plane graph}
In this section we show that when $G$ is a plane graph and $c$ is a positive coloring the vertex state sum formula for $\underline{\Delta}(G, c)$ has a topological interpretation, where a state is identified with a tuple of intersection points of the ${\alpha}$-curves and ${\beta}$-curves on a Heegaard diagram associated with $G$ and the Boltzmann weights come from the Fox calculus on the Heegaard surface. 

In Section 4.1 we give a simplified version of the vertex state sum for $\underline{\Delta}(G, c)$ in Prop. \ref{sum}. In Section 4.2 we construct a Heegaard diagram associated with $G$ and show a one-one correspondence between the states for $\underline{\Delta}(G, c)$ and the intersection points of ${\alpha}$-curves and ${\beta}$-curves on the Heegaard diagram. In Section 4.3 we study a vertex state sum formula for $\Delta_{(G, c)}(t)$ and compare it with that for $\underline{\Delta}(G, c)$.

\subsection{Vertex state sum formula for $\underline{\Delta}(G, c)$ when $G$ is a plane graph}
For a plane graph $G$, when all the multiplicities of the coloring $c$ are positive and the framing is the blackboard framing, we can simplify the vertex state sum for $\underline{\Delta}(G, c)$ in Prop. \ref{sum} as follows. 

Choose an initial point $\delta$ on an outermost edge of $G$. A {\it state} is a map $s: E \to \{0, 1\}$ which sends the edge containing $\delta$ to zero and satisfies the condition that at each vertex, the sum of $s(e)$ for all the edges $e$ pointing toward the vertex equals that for the edges pointing out of the vertex. Let $\mathcal{S}$ be the set of states. By definition, around each vertex we have the following six possibilities under a state, where the dotted edges are those which are sent to zero and the solid edges are those which are sent to one.
\begin{align*}
\begin{tikzpicture}[baseline=-0.65ex, thick, scale=1.2]
\draw [dotted] (0, 0) [->-] to (0,0.5);
\draw  [dotted] (0,0.5) [->]  to  (0.66,1);
\draw  [dotted] (0,0.5) [->]  to  (-0.66,1);
\end{tikzpicture} \quad
\begin{tikzpicture}[baseline=-0.65ex, thick, scale=1.2]
\draw  (0, 0) [->-] to (0,0.5);
\draw   (0,0.5) [->]  to  (0.66,1);
\draw  [dotted] (0,0.5) [->]  to  (-0.66,1);
\end{tikzpicture} \quad
\begin{tikzpicture}[baseline=-0.65ex, thick, scale=1.2]
\draw  (0, 0) [->-] to (0,0.5);
\draw  [dotted] (0,0.5) [->]  to  (0.66,1);
\draw   (0,0.5) [->]  to  (-0.66,1);
\end{tikzpicture} \hspace{1cm}
\begin{tikzpicture}[baseline=-0.65ex, thick, scale=1.2]
\draw [dotted] (0, 0.5) [->] to (0,1);
\draw  [dotted] (0,0.5) [-<-]  to  (0.66,0);
\draw [dotted]  (0,0.5) [-<-]  to  (-0.66,0);
\end{tikzpicture} \quad
\begin{tikzpicture}[baseline=-0.65ex, thick, scale=1.2]
\draw  (0, 0.5) [->] to (0,1);
\draw   (0,0.5) [-<-]  to  (0.66,0);
\draw  [dotted] (0,0.5) [-<-]  to  (-0.66,0);
\end{tikzpicture} \quad
\begin{tikzpicture}[baseline=-0.65ex, thick, scale=1.2]
\draw  (0, 0.5) [->] to (0,1);
\draw  [dotted] (0,0.5) [-<-]  to  (0.66,0);
\draw   (0,0.5) [-<-]  to  (-0.66,0);
\end{tikzpicture}
\end{align*}
The solid edges in a given state $s$ of $G$ constitute a collection of simple closed curves. We call them {\it{solid curves}} for $s$. Define the sign of a state $s$ to be $$\mathrm{sign}(s)=(-1)^{\text{the number of solid curves for $s$}}.$$ 

From $G$ we can obtain a collection of oriented simple closed curves $L_{(G, c)}$ by the local transformation depicted in Figure \ref{fig:e24}, where we replace each edge of color $j$ by $j$ parallel strands. We assume that each curve of $L_{(G, c)}$ is colored by $1$.

\begin{figure}
\begin{tikzpicture}[baseline=-0.65ex, thick, scale=0.8]
\draw (0,-1)  [->]-- (0,1);
\draw (0.3, 0.8) node {$4$} ;
\draw (1, 0) node {$\Longrightarrow$} ;
\draw (2,-1)  [->]-- (2,1);
\draw (2.5,-1)  [->]-- (2.5,1);
\draw (3,-1)  [->]-- (3,1);
\draw (3.5,-1)  [->]-- (3.5,1);
\end{tikzpicture}\quad \quad\quad
\begin{tikzpicture}[baseline=-0.65ex, thick, scale=0.8]
\draw (0,-1)  [->-]-- (0,0);
\draw (0,0)  [->]-- (1,1);
\draw (0,0)  [->]-- (-1,1);
\draw (1.3, 0.8) node {$3$} ;
\draw (-1.3, 0.8) node {$1$} ;
\draw (0.3, -0.8) node {$4$} ;
\end{tikzpicture}$\Longrightarrow$
\begin{tikzpicture}[baseline=-0.65ex, thick, scale=0.8]
\draw (1.2,-1) to [out=90, in=270] (1.2,0) [->] to [out=90,in=225] (2,1);
\draw (0.7,-1) to [out=90, in=270] (0.7,0) [->] to [out=90,in=225] (1.5,1);
\draw (-0.3,-1) to [out=90, in=270] (-0.3,0) [->] to [out=90,in=315] (-1.1,1);
\draw (0.2,-1) to [out=90, in=270] (0.2,0) [->] to [out=90,in=225] (1,1);
\end{tikzpicture}
	\caption{The way we obtain the link $L_{(G, c)}$ from $G$.}
	\label{fig:e24}
\end{figure}
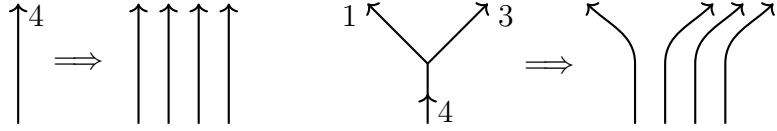

Define $\mathrm{Rot}(G, c, \delta)$ to be the sum of the rotation numbers of those simple closed curves in $L_{(G, c)}$ which are disjoint with $\delta$. Here the rotation number of a simple closed curve is defined to be $1$ (resp. $-1$) if it is clockwise (resp. counter-clockwise).  

We define the weight of a state at a vertex $v$ as in Table \ref{table1}. Then we get the following state sum formula for $\underline{\Delta}(G, c)$.

\begin{table}
\begin{tabular} {|c|c|c|c|c|c|c|} \hline
    State & \begin{tikzpicture}[baseline=-0.65ex, thick]
\draw [dotted] (0, -0.5) [->-] to (0,0);
\draw  [dotted] (0,0) [->]  to  (0.66,0.5);
\draw  [dotted] (0,0) [->]  to  (-0.66,0.5);
\draw (-0.66, 0.75) node {$i$};
\draw (0.66, 0.75) node {$j$};
\draw (0, -0.75) node {$i+j$};
\end{tikzpicture}
     & \begin{tikzpicture}[baseline=-0.65ex, thick]
\draw  (0, -0.5) [->-] to (0,0);
\draw  (0,0) [->]  to  (0.66,0.5);
\draw  [dotted] (0,0) [->]  to  (-0.66,0.5);
\draw (-0.66, 0.75) node {$i$};
\draw (0.66, 0.75) node {$j$};
\draw (0, -0.75) node {$i+j$};
\end{tikzpicture}
      & \begin{tikzpicture}[baseline=-0.65ex, thick]
\draw  (0, -0.5) [->-] to (0,0);
\draw  [dotted] (0,0) [->]  to  (0.66,0.5);
\draw   (0,0) [->]  to  (-0.66,0.5);
\draw (-0.66, 0.75) node {$i$};
\draw (0.66, 0.75) node {$j$};
\draw (0, -0.75) node {$i+j$};
\end{tikzpicture}
& \begin{tikzpicture}[baseline=-0.65ex, thick]
\draw [dotted] (0, 0) [->] to (0,0.5);
\draw  [dotted] (0,0) [-<-]  to  (0.66,-0.5);
\draw [dotted]  (0,0) [-<-]  to  (-0.66,-0.5);
\draw (-0.66, -0.75) node {$i$};
\draw (0.66, -0.75) node {$j$};
\draw (0, 0.75) node {$i+j$};
\end{tikzpicture}
& \begin{tikzpicture}[baseline=-0.65ex, thick]
\draw  (0, 0) [->] to (0,0.5);
\draw   (0,0) [-<-]  to  (0.66,-0.5);
\draw [dotted]  (0,0) [-<-]  to  (-0.66,-0.5);
\draw (-0.66, -0.75) node {$i$};
\draw (0.66, -0.75) node {$j$};
\draw (0,0.75) node {$i+j$};
\end{tikzpicture}
& \begin{tikzpicture}[baseline=-0.65ex, thick]
\draw  (0, 0) [->] to (0,0.5);
\draw  [dotted] (0,0) [-<-]  to  (0.66,-0.5);
\draw   (0,0) [-<-]  to  (-0.66,-0.5);
\draw (-0.66, -0.75) node {$i$};
\draw (0.66, -0.75) node {$j$};
\draw (0, 0.75) node {$i+j$};
\end{tikzpicture}
       \\ \hline 
  \vspace{-2mm}     &&&&&&\\
$Wt(v; s)$ & $\{2i+2j\}_q$   & $\{2j\}_q$ & $\{2i\}_qq^{-2j}$ &$1$ &$q^{2i}$&$1$\\ 
 &&&&&&\\
\hline 
\end{tabular}
\vspace{3mm}
\caption{The weight of a state at a vertex for $\underline{\Delta}(G, c)$.}
\label{table1}
\end{table}

\begin{prop} 
\label{vsum}
For a plane graph $G$ with the blackboard framing and a positive coloring $c$, we have
\begin{equation}
\underline{\Delta}(G, c)=\frac{ q^{2\mathrm{Rot}(G, c, \delta)} }{q^{2j}-q^{-2j}}\sum_{s\in \mathcal{S}}\mathrm{sign}(s) \prod_{v\in V}Wt(v;s),
\label{eq4}\tag{4}
\end{equation}
where $j$ is the color of the edge containing $\delta$.
\end{prop}
\begin{proof}
We can assume that the edges around each vertex of $G$ point upward after an isotopy of $\mathbb{R}^2$ and that the edge containing $\delta$ is a leftmost edge of $G$. Namely $G$ is in the following position around a vertex. 
\begin{align*}
\begin{tikzpicture}[baseline=-0.65ex, thick, scale=1]
%\draw  (0,-1) [->-] to (0, 0.33);
\draw  (0, 0) [->-] to (0,0.5);
%\draw  (1,-1) [->-] to (1, -0.33);
%\draw  (1, -0.33) [->] to (1,1);
\draw   (0,0.5) [->]  to  (0.66,1);
\draw   (0,0.5) [->]  to  (-0.66,1);
\end{tikzpicture} \hspace{1cm} \text{or} \hspace{1cm} 
\begin{tikzpicture}[baseline=-0.65ex, thick, scale=1]
%\draw  (0,-1) [->-] to (0, 0.33);
\draw  (0, 0.5) [->] to (0,1);
%\draw  (1,-1) [->-] to (1, -0.33);
%\draw  (1, -0.33) [->] to (1,1);
\draw   (0,0.5) [-<-]  to  (0.66,0);
\draw   (0,0.5) [-<-]  to  (-0.66,0);
\end{tikzpicture}
\end{align*}
The left-hand side of (\ref{eq4}), as a topological invariant, does not change under the isotopy. We show that for such a graph diagram the vertex state sum of $\underline{\Delta}(G, c)$ that we stated in Prop. \ref{sum} coincides with the right-hand side of (\ref{eq4}).

Recall that a state in Prop. \ref{sum} is defined to be an assignment of $e_0$ to the edge containing $\delta$ and $e_0$ or $e_1$ to any of the other edges. Therefore $\mathcal{S}$ is a subset of the set of states in Prop. \ref{sum}. From the Boltzmann weights in Table 2 we see that any state which is not in $\mathcal{S}$ has no contribution to $\underline{\Delta}(G, c)$. Therefore we only need to consider the states in $\mathcal{S}$. 

Since $G$ is a plane graph with blackboard framing, there is no crossings and half-twist symbols on the diagram. We only need to consider the Boltzmann weights at vertices and at critical points.
For $s\in \mathcal{S}$, the product of Boltzmann weights at the vertices is exactly $\prod_{v\in V}Wt(v;s)$. To finish the proof we need to show that the product of the Boltzmann weights at the critical points is $$\mathrm{sign}(s)q^{2[\mathrm{Rot} (G, c, \delta)-(-1)^{\theta}j]},$$ where $\theta$ is defined in Prop. \ref{sum}. 

Note that for an oriented simple closed curve $U$ in $\mathbb{R}^2$ colored by a positive integer $j$, if we assign $e_i$ to it, the product of the Boltzmann weights at its critical points is $(-1)^{i}q^{\mathrm{Rot}(U)\cdot 2j}$ for $i=0, 1$, where $\mathrm{Rot}(U)$ is the rotation number of $U$. Only the sign $(-1)^{i}$ depends on the state which is $1$ for dotted curve (assignment of $e_0$) and $1$ for solid curve (assignment of $e_1$), while the $q$-power $q^{\mathrm{Rot}(U)\cdot 2j}$ only depends on the coloring. 

For a state $s\in \mathcal{S}$ of $G$, the product of the Boltzmann weights at all critical points consists of two factors: the sign, which is given by the number of solid curves of $s$ and thus is $\mathrm{sign}(s)$, and the $q$-power, which is $q^{2[\mathrm{Rot} (G, c, \delta)-(-1)^{\theta}j]}$ as we see below. 

Recall that $L_{(G, c)}$ is a collection of simple closed curves colored by $1$. Let $L_{(G, c)}$ be in the Morse position inherited from that of $G$. Then the $q$-power factor of the product of the Boltzmann weights at all the critical points of $L_{(G, c)}$ is $$\prod_{U\in L_{(G, c)} }q^{2\mathrm{Rot}(U)}=q^{\sum_{U\in L_{(G, c)}}2\mathrm{Rot}(U)} =q^{2\mathrm{Rot} (G, c, \delta)-(-1)^{\theta}2j},$$ where $-(-1)^{\theta}2j$ is the sum of the rotation numbers of curves of $L_{(G, c)}$ that pass $\delta$. Since $\delta$ is on a leftmost edge of $G$, $-(-1)^{\theta}2j$ is $2j$ when the edge points upward around $\delta$, and $-2j$ otherwise. From the construction of $L_{(G, c)}$ and the Boltzmann weights defined for critical points, we see that $q^{2\mathrm{Rot} (G, c, \delta)-(-1)^{\theta}2j}$ is the $q$-power factor of the product of the Boltzmann weights at all the critical points of $G$ as well.

After multiplying by $\frac{q^{(-1)^{\theta}2j}}{q^{2j}-q^{-2j}}$, as in the last step of Prop \ref{sum}, we get the desired equality. 
\end{proof}

\subsection{A Heegaard diagram for $G$.}

For an oriented plane graph $G$ without source or sink in $\mathbb{R}^2$, we introduce a Heegaard diagram, which is inspired by the Heegaard diagram in \cite[Section 5]{MR2529302}. The construction is as follows. Let $S^2=\mathbb{R}^2\cup \{\infty\}$.
\begin{enumerate}
\item The Heegaard surface is $S^{2}$ where $G$ is embedded as a plane graph. 
\item Choose an initial point $\delta$ on an edge of $G$. We regard this initial point as a new vertex of $G$.
\item At each vertex, introduce a base point $w$, and on each edge of $G$ introduce a base point $z$. Let $\boldsymbol{w}$ and $\boldsymbol{z}$ be the set of $w$'s and $z$'s respectively.
\item Around each vertex $v$, introduce a curve $\alpha_v$ which encloses the base point $w$ at $v$ and the base point(s) $z$('s) on the edge(s) pointing to $v$. Introduce a curve $\beta_v$ which encloses the base point $w$ at $v$ and the base point(s) $z$('s) on the edge(s) pointing out of $v$. 
\item Remove the $\alpha$-curve and $\beta$-curve around $\delta$.
\end{enumerate}
As a result, we get the following data $H=(S^{2}, \{\alpha_v\}_{v\in V}, \{\beta_v\}_{v\in V}, \boldsymbol{w}, \boldsymbol{z})$, where $V$ is the set of vertices of $G$ not including $\delta$. See Fig. \ref{fig:f1} for an example of the Heegaard diagram. 

The data $H$ constructed above is a Heegaard diagram for $(G, \delta)$. To be precise, from $(G, \delta)$ we construct a new graph $\widetilde{G}$ as below by inserting a thick edge at each vertex and at $\delta$ and splitting the vertex into two vertices.
\begin{align*}
\begin{tikzpicture}[baseline=-0.65ex, thick, scale=1]
%\draw  (0,-1) [->-] to (0, 0.33);
\draw  (0, 0) [->-] to (0,0.5);
%\draw  (1,-1) [->-] to (1, -0.33);
%\draw  (1, -0.33) [->] to (1,1);
\draw   (0,0.5) [->]  to  (0.66,1);
\draw   (0,0.5) [->]  to  (-0.66,1);
\draw (1.25, 0.5) node {$\Rightarrow$};
\end{tikzpicture} 
\begin{tikzpicture}[baseline=-0.65ex, thick, scale=1]
\draw  (0, -0.25) [->-] to (0,0.25);
\draw   (0,0.75) [->]  to  (0.66,1.25);
\draw   (0,0.75) [->]  to  (-0.66,1.25);
\draw [ultra thick] (0, 0.25) [->-] to (0,0.75);
\end{tikzpicture}
\hspace{1cm}
\begin{tikzpicture}[baseline=-0.65ex, thick, scale=1]
%\draw  (0,-1) [->-] to (0, 0.33);
\draw  (0, 0.5) [->] to (0,1);
%\draw  (1,-1) [->-] to (1, -0.33);
%\draw  (1, -0.33) [->] to (1,1);
\draw   (0,0.5) [-<-]  to  (0.66,0);
\draw   (0,0.5) [-<-]  to  (-0.66,0);
\draw (1.25, 0.5) node {$\Rightarrow$};
\end{tikzpicture}
\begin{tikzpicture}[baseline=-0.65ex, thick, scale=1]
%\draw  (0,-1) [->-] to (0, 0.33);
\draw  (0, 0.75) [->] to (0,1.25);
%\draw  (1,-1) [->-] to (1, -0.33);
%\draw  (1, -0.33) [->] to (1,1);
\draw   (0,0.25) [-<-]  to  (0.66,-0.25);
\draw   (0,0.25) [-<-]  to  (-0.66,-0.25);
\draw [ultra thick] (0, 0.25) [->-] to (0,0.75);
\end{tikzpicture}
\hspace{1cm}
\begin{tikzpicture}[baseline=-0.65ex, thick, scale=1]
\draw  (0, 0) [->] to (0,1);
\draw (-0.3, 0.5) node{$\delta$};
\draw [fill] (0, 0.5) circle [radius=.05];
\draw (0.75, 0.5) node {$\Rightarrow$};
\end{tikzpicture}\hspace{0.25cm}
\begin{tikzpicture}[baseline=-0.65ex, thick, scale=1]
\draw  (0, -0.25) [->-] to (0,0.25);
\draw  [ultra thick] (0, 0.25) [->-] to (0,0.75);
\draw  (0, 0.75) [->] to (0,1.25);
\end{tikzpicture}
\end{align*}
Then $\widetilde{G}$ becomes a balanced bipartite graph with these thick edges. Then we can check that $H$ is a Heegaard diagram for $\widetilde{G}$ (see \cite[Definition 3.1]{bao}). For simplicity we suppress the thick edges and say that $H$ is a Heegaard diagram for $(G, \delta)$.

Consider the intersection points of $\alpha$-curves and $\beta$-curves. Let $$\mathcal{T}:=\{\{x_v\}_{v\in V}\vert x_v\in \alpha_{\sigma(v)} \cap \beta_v \text{ for a bijection $\sigma$ of $V$}\}.$$
Let $\sigma_x$ be the bijection corresponding to $x\in \mathcal{T}$. Since each $\alpha$-curve and $\beta$-curve is a simple closed curve in $\mathbb{R}^2$, we call the disk bounded by an $\alpha$-curve (resp. $\beta$-curve) an {\it $\alpha$-disk} (resp. {\it $\beta$-disk}). From the construction we see that $\alpha$-disks (resp. $\beta$-disks) are disjoint from each other, while the intersection of $\alpha$-disks and $\beta$-disks are disjoint union of disks, which we call {\it bigons}. Each bigon contains a base point (either $z$ or $w$) in its interior and contains on its boundary two intersection points of an $\alpha$-curve and a $\beta$-curve.
Two elements $x=\{x_v\}_{v\in V}$ and $y=\{y_v\}_{v\in V}$ are said to be equivalent ($x \sim y$) if $x_v$ and $y_v$ belong to the same bigon for any $v\in V$. Let $[x]$ be the equivalence class of $x$. It is obvious that $\sigma_x$ keeps invariant within the equivalence class of $x$.

\begin{figure}
	\centering
		\includegraphics[width=0.6\textwidth]{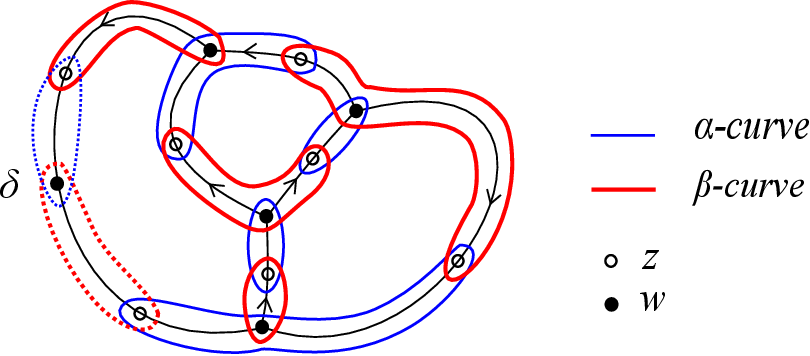}
	\caption{A Heegaard diagram for $G$ with an initial point $\delta$.}
	\label{fig:f1}
\end{figure}

\begin{prop}
\label{bijection}
There is a bijection $\psi: \mathcal{T}/ \mathord{\sim} \to \mathcal{S}$. 
\end{prop}
\begin{proof}
The map $\psi: \mathcal{T}/ \mathord{\sim} \to \mathcal{S}$ is constructed as follows. Let $z_e\in \boldsymbol{z}$ be the base point on an edge $e$. Then for $x=\{x_v\}_{v\in V}$, $\psi([x])$
sends $e$ to $1$ if there is $x_v\in x$ belonging to the bigon containing $z_e$. The other edges are sent to zero. 

To prove that $\psi([x])$ is an element of $\mathcal{S}$, it is enough to show that the following situations can not occur under $\psi([x])$.
Around an even vertex $v$, the following cases can not appear under $\psi([x])$:
\begin{align*}
\begin{tikzpicture}[baseline=-0.65ex, thick, scale=1.2]
\draw  (0, 0.5) [->] to (0,1);
\draw   (0,0.5) [-<-]  to  (0.66,0);
\draw   (0,0.5) [-<-]  to  (-0.66,0);
\end{tikzpicture} \quad\quad
\begin{tikzpicture}[baseline=-0.65ex, thick, scale=1.2]
\draw [dotted] (0, 0.5) [->] to (0,1);
\draw   (0,0.5) [-<-]  to  (0.66,0);
\draw   (0,0.5) [-<-]  to  (-0.66,0);
\end{tikzpicture} \quad\quad
\begin{tikzpicture}[baseline=-0.65ex, thick, scale=1.2]
\draw [dotted] (0, 0.5) [->] to (0,1);
\draw  [dotted] (0,0.5) [-<-]  to  (0.66,0);
\draw   (0,0.5) [-<-]  to  (-0.66,0);
\end{tikzpicture} \quad\quad
\begin{tikzpicture}[baseline=-0.65ex, thick, scale=1.2]
\draw [dotted] (0, 0.5) [->] to (0,1);
\draw   (0,0.5) [-<-]  to  (0.66,0);
\draw  [dotted] (0,0.5) [-<-]  to  (-0.66,0);
\end{tikzpicture} \quad\quad
\begin{tikzpicture}[baseline=-0.65ex, thick, scale=1.2]
\draw  (0, 0.5) [->] to (0,1);
\draw  [dotted] (0,0.5) [-<-]  to  (0.66,0);
\draw  [dotted] (0,0.5) [-<-]  to  (-0.66,0);
\end{tikzpicture}.
\end{align*}
The first and second cases can not appear since $\alpha_v \cap x$ has a unique element. The third and fourth cases do not appear since $\beta_v \cap x$ is not empty. The fifth case can not appear since $\alpha_v \cap x$ is not empty.  

Around an odd vertex $v$, the following cases can not appear under $\psi([x])$ for similar reasons:
\begin{align*}
\begin{tikzpicture}[baseline=-0.65ex, thick, scale=1.2]
\draw  (0, 0) [->-] to (0,0.5);
\draw   (0,0.5) [->]  to  (0.66,1);
\draw   (0,0.5) [->]  to  (-0.66,1);
\end{tikzpicture} \quad\quad
\begin{tikzpicture}[baseline=-0.65ex, thick, scale=1.2]
\draw [dotted] (0, 0) [->-] to (0,0.5);
\draw   (0,0.5) [->]  to  (0.66,1);
\draw   (0,0.5) [->]  to  (-0.66,1);
\end{tikzpicture} \quad\quad
\begin{tikzpicture}[baseline=-0.65ex, thick, scale=1.2]
\draw  [dotted](0, 0) [->-] to (0,0.5);
\draw  [dotted] (0,0.5) [->]  to  (0.66,1);
\draw   (0,0.5) [->]  to  (-0.66,1);
\end{tikzpicture} \quad\quad
\begin{tikzpicture}[baseline=-0.65ex, thick, scale=1.2]
\draw  [dotted](0, 0) [->-] to (0,0.5);
\draw   (0,0.5) [->]  to  (0.66,1);
\draw  [dotted] (0,0.5) [->]  to  (-0.66,1);
\end{tikzpicture} \quad\quad
\begin{tikzpicture}[baseline=-0.65ex, thick, scale=1.2]
\draw  (0, 0) [->-] to (0,0.5);
\draw  [dotted] (0,0.5) [->]  to  (0.66,1);
\draw [dotted]  (0,0.5) [->]  to  (-0.66,1);
\end{tikzpicture}.
\end{align*}

To show that $\psi$ is a bijection, we need to construct its inverse map $\varphi: \mathcal{S} \to \mathcal{T}/ \mathord{\sim}$. 
Consider a state $s\in \mathcal{S}$. If a vertex $v$ belongs to a solid curve $\gamma$, let $e$ be the solid edge pointing from $v$ to another vertex $v'$ on $\gamma$. Then let $x_v\in \alpha_{v'}\cap \beta_v$ be on the bigon containing $z_e$. If a vertex $v$ does not belong to any solid curve, let $x_v \in \alpha_{v}\cap \beta_v$ be on the bigon containing the base point $w_v$ at $v$. We can check that each $\alpha_v$ and $\beta_v$ is occupied exactly once. Therefore $x=\{x_v\}_{v\in V}\in \mathcal{T}$. Since $[x]$ does not depend on which intersection point we choose on a bigon, $\varphi$ is a well-defined map. We see that both $\psi\circ \varphi$ and $\varphi\circ\psi$ are identity maps.

\end{proof}

\begin{lemma}
\label{signeq}
Under the bijection $\psi$ in Prop. \ref{bijection} we have $$\mathrm{sign}(\psi([x]))=\mathrm{sign}(\sigma_x)\prod_{v\in V}(-1)^{\delta_{v, \sigma_x(v)}+1},$$ for any $x\in \mathcal{T}$, where $\delta_{v, \sigma_x(v)}$ is the Kronecker delta.
\end{lemma}
\begin{proof}
Each solid curve for $\psi([x])$ consisting of $k$ edges corresponds to a $k$-cycle of $\sigma_x$. Its contribution to $\mathrm{sign}(\sigma_x)$ is $(-1)^{k-1}$. On the other hand, each solid edge contributes  minus one to $\prod_{v\in V}(-1)^{\delta_{v, \sigma_x(v)}+1}$. Therefore 
\begin{eqnarray*}
\mathrm{sign}(\sigma_x)\prod_{v\in V}(-1)^{\delta_{v, \sigma_x(v)}+1}&=&\prod_{\text{solid curve of length $k$}, \,\, k\in \mathbb{N}} (-1)^{k-1} (-1)^{k}\\
&=& \prod_{\text{solid curve of length $k$}, \,\, k\in \mathbb{N}} (-1)\\
&=& (-1)^{\text{the number of solid curves for $\psi([x])$}}=\mathrm{sign}(\psi([x])).
\end{eqnarray*}
\end{proof}

%%%%%%%%%%%%%%%%%%%%%

\subsection{A vertex state sum formula for $\Delta_{(G, c)}(t)$}
The purpose of this section is to prove Proposition \ref{weaker}.
Let $G$ be a connected plane graph with a positive coloring $c$. Let $X$ be the complement of $G$ in $S^3$. Recall that in Section 3 we defined a map $\phi_c: \pi_1(X, x_0)\to \mathbb{Z}\langle t\rangle$ which sends an oriented meridian of an edge $e$ to $t^{c(e)}$. Let $\mathbb{Z}\pi_1(X, x_0)$ and $\mathbb{Z}[t, t^{-1}] $ be the group rings of $\pi_1(X, x_0)$ and $\mathbb{Z}\langle t\rangle$ respectively. For convenience, we consider a large ring $\mathbb{Z}[t^{1/2}, t^{-1/2}]\supset \mathbb{Z}[t, t^{-1}]$. Then $\phi_c$ naturally induces a ring homomorphism, which we still call $\phi_c$:
\begin{equation}
\phi_c: \mathbb{Z}\pi_1(X, x_0) \to \mathbb{Z}[t^{1/2}, t^{-1/2}]. \tag{5}
\end{equation}

Consider the Heegaard diagram $H$ constructed in Section 4.2. Following \cite[Section 5.3]{bao} or \cite[Section 5]{MR2805998}, we can calculate a determinant $\det (\phi_c(\frac{\partial \beta_v}{\partial \alpha_u}))_{u, v\in V}$ using Fox calculus. For a general sutured manifold, \cite{MR2805998} studied the Euler characteristic of the sutured Floer homology. In \cite[Section 5]{bao}, we explained the theory for the concrete case of bipartite graphs. 

We leave the theoretical discussions to \cite[Sections 4, 5]{MR2805998} and explain below how to calculate the determinant $\det (\phi_c(\frac{\partial \beta_v}{\partial \alpha_u}))_{u, v\in V}$.

\begin{enumerate}
\item Note that each $\alpha$-curve or $\beta$-curve bounds an $\alpha$-disk or $\beta$-disk in $\mathbb{R}^2$. Choose an orientation of $\mathbb{R}^2$ so that each $\alpha$-disk or $\beta$-disk induces counter-clockwise orientation on its boundary, which is the given $\alpha$-curve or $\beta$-curve.
\item For an edge $e$, let $c_e$ be an oriented simple arc on $\mathbb{R}^2$ connecting the base point $z_e$ to the base point $w$ at the vertex that $e$ points to. 
\item For each $\beta$-curve $\beta_v$, consider its intersection with $\alpha$-curves and $c_e$'s. If an intersection point between $\alpha_u$ (resp. $c_e$) and $\beta_v$ has positive sign, record it by $\alpha_u$ (resp. $c_e$), and $\alpha_u^{-1}$ (resp. $c_e^{-1}$) if the sign is negative. In this way, we write $\beta_v$ as a word, which is a product of elements of $\alpha_u$, $\alpha_u^{-1}$, $c_e$ and $c_e^{-1}$ for $u\in V$ and $e\in E$.
\item Calculate the Fox calculus $\frac{\partial \beta_v}{\partial \alpha_u}$ for any $u, v\in V$, which becomes a finite sum of words in $\alpha_w$, $\alpha_w^{-1}$, $c_e$ and $c_e^{-1}$ for $w\in V$ and $e\in E$. See \cite[Chapter 7]{MR1417494} for the definition and an application of Fox calculus. 
\item Let $\phi_c(\alpha_w)=\phi_c(\alpha_w^{-1})=1$, $\phi_c(c_e)=t^{c(e)}$ and $\phi_c(c_e^{-1})=t^{-c(e)}$. Then $\phi_c(\frac{\partial \beta_v}{\partial \alpha_u})$ becomes an element of $\mathbb{Z}[t^{1/2}, t^{-1/2}]$. The determinant $\det (\phi_c(\frac{\partial \beta_v}{\partial \alpha_u}))_{u, v\in V}$ is thus an element of $\mathbb{Z}[t^{1/2}, t^{-1/2}]$ as well. 
\end{enumerate}

Note that $\det (\phi_c(\frac{\partial \beta_v}{\partial \alpha_u}))_{u, v\in V}$ is only well-defined modulo $\pm \mathbb{Z}\langle t^{1/2} \rangle$, where $\mathbb{Z}\langle t^{1/2} \rangle=\{t^{k/2}\mid k\in \mathbb{Z}\}$. For three Laurent polynomials $f(t), g(t), h(t)\neq 0$ in $\mathbb{Z}[t^{1/2}, t^{-1/2}]$, we denote $$\frac{f(t)}{h(t)}\overset{\cdot}{=} \frac{g(t)}{h(t)}$$ if the two sides are equal modulo $\pm \mathbb{Z}\langle t^{1/2} \rangle$. We have the following lemma.

\begin{lemma}
\label{lA}
Let $G$ be a connected plane graph with a positive coloring $c$. We have
\begin{equation*}
(t^{1/2}-t^{-1/2})^{\vert V\vert -1} \Delta_{(G, c)}(t) \overset{\cdot}{=}\frac{1}{t^{j/2}-t^{-j/2}}\det (\phi_c(\frac{\partial \beta_v}{\partial \alpha_u}))_{u, v\in V},
\end{equation*}
where $j$ is the color of the edge containing the initial point $\delta$. 
\end{lemma}
\begin{proof}
By the definition of $\Delta_{(G, c)}(t)$ in \cite[Definition 3.5]{alex}, we have $$\Delta_{(G, c)}(t)\overset{\cdot}{=} \frac{\langle G \vert \delta \rangle }{(t^{1/2}-t^{-1/2})^{\vert V\vert -1} (t^{j/2}-t^{-j/2})},$$ 
where the state sum $\langle G \vert \delta \rangle$ defined in \cite[(1)]{alex} coincides with $\phi_c(\Delta_{\widetilde{G}})$ defined in \cite[(2)]{bao}. Therefore $\langle G \vert \delta \rangle$ and $\det (\phi_c(\frac{\partial \beta_v}{\partial \alpha_u}))_{u, v\in V}$ both calculate the Euler characteristic of the Heegaard Floer homology of $\widetilde G$. Namely we have $\langle G \vert \delta \rangle \overset{\cdot}{=} \det (\phi_c(\frac{\partial \beta_v}{\partial \alpha_u}))_{u, v\in V}$, which implies the statement of the lemma. 
\end{proof}

\vspace{3mm}
 Our next step is to calculate $\det (\phi_c(\frac{\partial \beta_v}{\partial \alpha_u}))_{u, v\in V}$.
Around an even vertex $v$ with the following coloring
\begin{figure}[h!]
	\centering
		\includegraphics[width=0.5\textwidth]{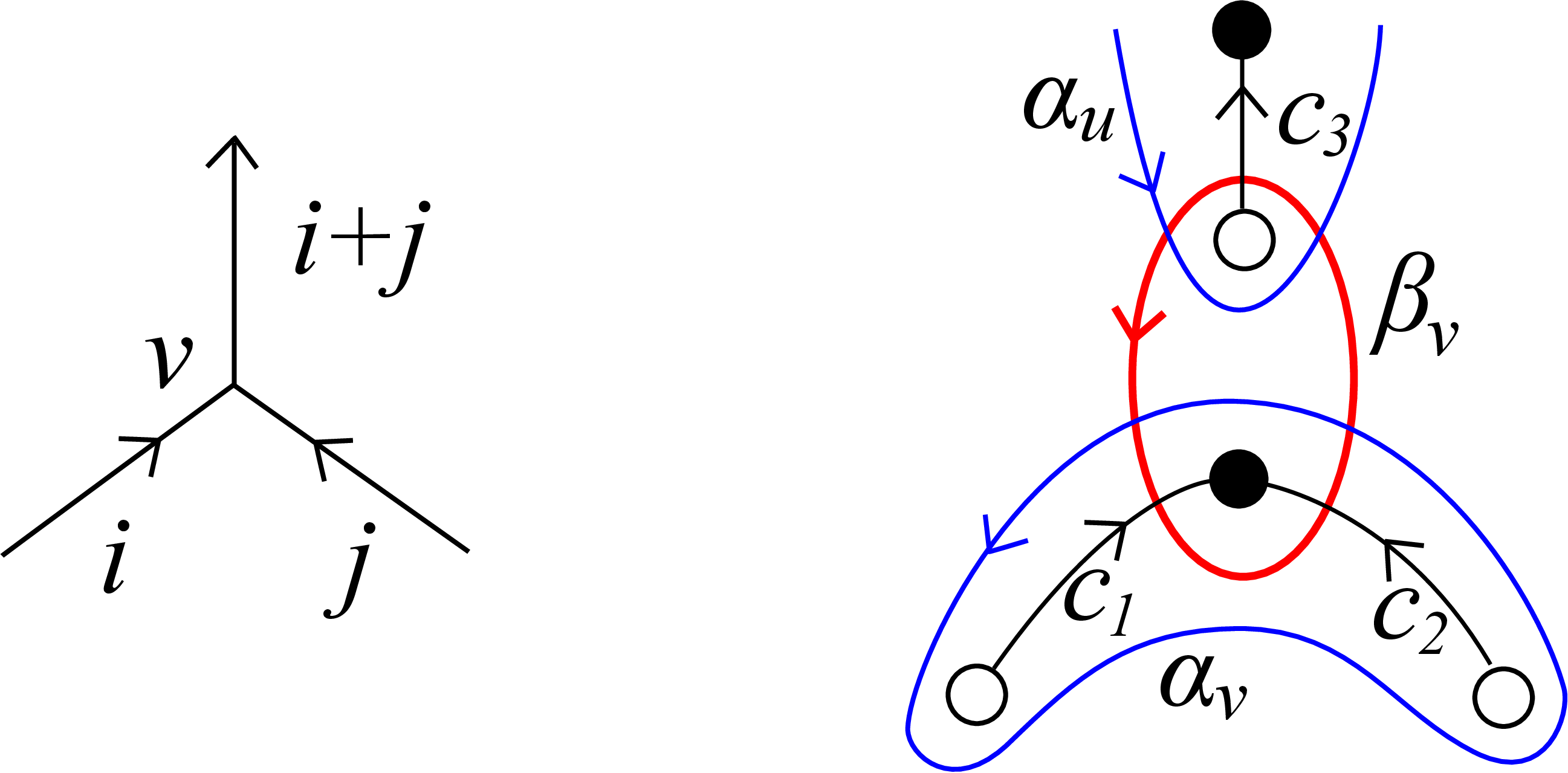}
	\label{fig:f2}
\end{figure}
we have
$\beta_v=\alpha_u c_3 \alpha_u^{-1}\alpha_v c_1^{-1}c_2^{-1}\alpha_v^{-1}$, where $c_1$, $c_2$ and $c_3$ are $c_e$'s for $e$ being the edge with color $i$, $j$ and $i+j$ respectively.
Therefore
\begin{eqnarray*}
\phi_c(\frac{\partial \beta_v}{\partial \alpha_u}) &\overset{\cdot}{=}&
\phi_c(1-\alpha_u c_3 \alpha_u^{-1})=1-\phi_c(c_3)=1-t^{i+j},\\
\phi_c(\frac{\partial \beta_v}{\partial \alpha_v})&\overset{\cdot}{=}&\phi_c(\alpha_u c_3 \alpha_u^{-1}(1-\alpha_v c_1^{-1}c_2^{-1}\alpha_v^{-1}))\\
&=&\phi_c(c_3)(1-\phi_c(c_1^{-1}c_2^{-1}))=t^{i+j}(1-t^{-i}t^{-j})=t^{i+j}-1.
\end{eqnarray*}
For any other $\alpha$-curve $\alpha_w$ disjoint with $\beta_v$, the Fox calculus $\frac{\partial \beta_v}{\partial \alpha_w}$ is zero. Multiplying each element in the row of $(\phi_c(\frac{\partial \beta_v}{\partial \alpha_u}))_{u, v\in V}$ that corresponds to $\beta_v$ by $t^{-(i+j)/2}$, the determinant does not change modulo $\mathbb{Z}\langle t^{1/2}\rangle$. Summarizing the discussion above we have
\begin{equation}
\label{eq6}
\phi_c(\frac{\partial \beta_v}{\partial \alpha_u}) \overset{\cdot}{=} \begin{cases}
\{\frac{i+j}{2}\}_t & u=v \\
-\{\frac{i+j}{2}\}_t & \text{$u$ and $v$ are joined by the edge with color $i+j$}\\
0 & \text{otherwise},
\end{cases}
\tag{6}
\end{equation}
where $\{k\}_t:=t^{k}-t^{-k}$ for $k\in \frac{1}{2}\mathbb{Z}$.

\vspace{3mm}

Around an odd vertex $v$ as below
\begin{figure}[h!]
	\centering
		\includegraphics[width=0.5\textwidth]{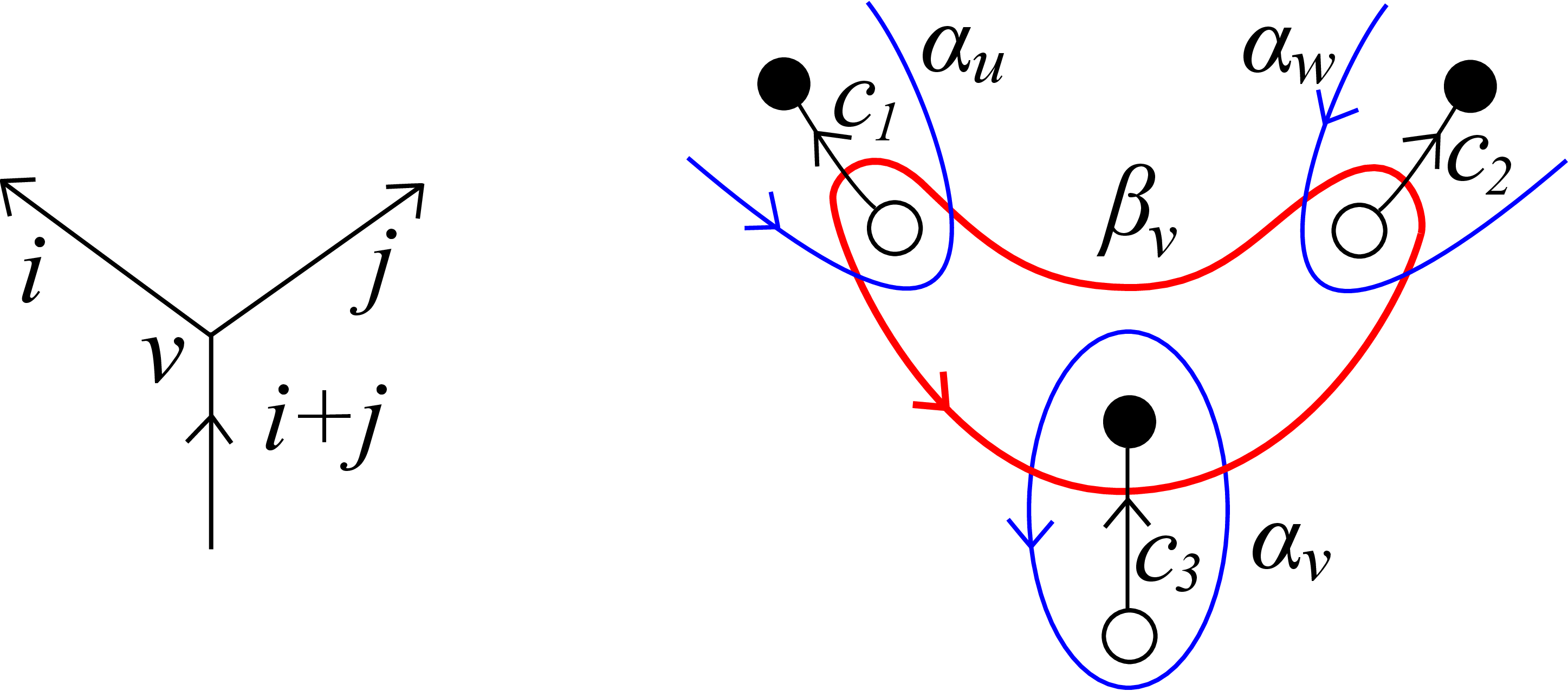}
	\label{fig:f3}
\end{figure}
we have $\beta_v=\alpha_vc_3^{-1}\alpha_v^{-1}\alpha_wc_2\alpha_w^{-1}\alpha_uc_1\alpha_u^{-1}$. 

Therefore
\begin{eqnarray*}
\phi_c(\frac{\partial \beta_v}{\partial \alpha_v}) &\overset{\cdot}{=}&
\phi_c(1-\alpha_v c_3^{-1} \alpha_v^{-1})=1-\phi_c(c_3^{-1})=1-t^{-(i+j)},\\
\phi_c(\frac{\partial \beta_v}{\partial \alpha_w})&\overset{\cdot}{=}&\phi_c(\alpha_v c_3^{-1} \alpha_v^{-1}(1-\alpha_w c_2\alpha_w^{-1}))\\
&=&\phi_c(c_3^{-1})(1-\phi_c(c_2))=t^{-(i+j)}(1-t^{j}),\\
\phi_c(\frac{\partial \beta_v}{\partial \alpha_u})&\overset{\cdot}{=}&\phi_c(\alpha_v c_3^{-1} \alpha_v^{-1}\alpha_wc_2\alpha_w^{-1}(1-\alpha_u c_1\alpha_u^{-1}))\\
&=&\phi_c(c_3^{-1}c_2)(1-\phi_c(c_1))=t^{-(i+j)}t^{j}(1-t^{i})=t^{-i}(1-t^{i}).
\end{eqnarray*}
For any other $\alpha$-curve $\alpha_r$ disjoint with $\beta_v$, the Fox calculus $\frac{\partial \beta_v}{\partial \alpha_r}$ is zero.
Multiplying each element in the row of $\phi_c(\frac{\partial \beta_v}{\partial \alpha_u}))_{u, v\in V}$ that corresponds to $\beta_v$ by $t^{(i+j)/2}$, the determinant does not change modulo $\mathbb{Z}\langle t^{1/2}\rangle$. Summarizing the discussion above we have
\begin{equation}
\label{eq7}
\phi_c(\frac{\partial \beta_v}{\partial \alpha_u}) \overset{\cdot}{=} \begin{cases}
\{\frac{i+j}{2}\}_t  & u=v \\
-t^{j/2}\{\frac{i}{2}\}_t & \text{$u$ and $v$ are joined by the edge with color $i$}\\
-t^{-i/2}\{\frac{j}{2}\}_t & \text{$u$ and $v$ are joined by the edge with color $j$}\\
0 & \text{otherwise}.
\end{cases}
\tag{7}
\end{equation}
Note that it is possible that $u$ and $v$ are joined by the two edges with colors $i$ and $j$. In this case $\phi_c(\frac{\partial \beta_v}{\partial \alpha_u})$ equals the sum of the second and the third terms.

\vspace{3mm}

Summarizing the discussion above we can prove the following lemma.

\begin{lemma}
\label{lB}
For a bijection $\sigma$ of $V$ and $v\in V$ we have
\begin{eqnarray*}
\phi_c(\frac{\partial \beta_v}{\partial \alpha_{\sigma(v)}}) \overset{\cdot}{=}  (-1)^{\delta_{v, \sigma(v)}+1} \sum_{[x]\in \mathcal{T}/\mathord{\sim}, \sigma_x=\sigma} \widetilde{Wt}(v, [x]),
\end{eqnarray*}
where $\widetilde{Wt}(v, [x])$ is defined in Table \ref{table2}.
\end{lemma}
\begin{proof}
Note that different elements in $\mathcal{T}/\mathord{\sim}$ may correspond to the same bijection $\sigma$. To get the right equality, we need to take the sum over all $[x]\in \mathcal{T}/\mathord{\sim}$ with $\sigma_x=\sigma$. The term $(-1)^{\delta_{v, \sigma(v)}+1}$ comes from the fact that if $v\neq \sigma(v)$ there is a minus sign in $\displaystyle \phi_c(\frac{\partial \beta_v}{\partial \alpha_{\sigma(v)}})$, as we can see in \eqref{eq6} and \eqref{eq7}.
\end{proof}

\begin{table}
\begin{tabular} {|c|c|c|c|c|c|c|} \hline
State $\psi([x])$ & \begin{tikzpicture}[baseline=-0.65ex, thick]
\draw [dotted] (0, -0.5) [->-] to (0,0);
\draw  [dotted] (0,0) [->]  to  (0.66,0.5);
\draw  [dotted] (0,0) [->]  to  (-0.66,0.5);
\draw (-0.66, 0.75) node {$i$};
\draw (0.66, 0.75) node {$j$};
\draw (0, -0.75) node {$i+j$};
\end{tikzpicture}
     & \begin{tikzpicture}[baseline=-0.65ex, thick]
\draw  (0, -0.5) [->-] to (0,0);
\draw  (0,0) [->]  to  (0.66,0.5);
\draw  [dotted] (0,0) [->]  to  (-0.66,0.5);
\draw (-0.66, 0.75) node {$i$};
\draw (0.66, 0.75) node {$j$};
\draw (0, -0.75) node {$i+j$};
\end{tikzpicture}
      & \begin{tikzpicture}[baseline=-0.65ex, thick]
\draw  (0, -0.5) [->-] to (0,0);
\draw  [dotted] (0,0) [->]  to  (0.66,0.5);
\draw   (0,0) [->]  to  (-0.66,0.5);
\draw (-0.66, 0.75) node {$i$};
\draw (0.66, 0.75) node {$j$};
\draw (0, -0.75) node {$i+j$};
\end{tikzpicture}
& \begin{tikzpicture}[baseline=-0.65ex, thick]
\draw [dotted] (0, 0) [->] to (0,0.5);
\draw  [dotted] (0,0) [-<-]  to  (0.66,-0.5);
\draw [dotted]  (0,0) [-<-]  to  (-0.66,-0.5);
\draw (-0.66, -0.75) node {$i$};
\draw (0.66, -0.75) node {$j$};
\draw (0, 0.75) node {$i+j$};
\end{tikzpicture}
& \begin{tikzpicture}[baseline=-0.65ex, thick]
\draw  (0, 0) [->] to (0,0.5);
\draw   (0,0) [-<-]  to  (0.66,-0.5);
\draw [dotted]  (0,0) [-<-]  to  (-0.66,-0.5);
\draw (-0.66, -0.75) node {$i$};
\draw (0.66, -0.75) node {$j$};
\draw (0,0.75) node {$i+j$};
\end{tikzpicture}
& \begin{tikzpicture}[baseline=-0.65ex, thick]
\draw  (0, 0) [->] to (0,0.5);
\draw  [dotted] (0,0) [-<-]  to  (0.66,-0.5);
\draw   (0,0) [-<-]  to  (-0.66,-0.5);
\draw (-0.66, -0.75) node {$i$};
\draw (0.66, -0.75) node {$j$};
\draw (0, 0.75) node {$i+j$};
\end{tikzpicture}
       \\ \hline
       \vspace{-2mm}& &&&&&\\
$\widetilde{Wt}(v; [x])$ & $\{\frac{i+j}{2}\}_t$ & $\{\frac{j}{2}\}_t t^{-i/2}$ & $\{\frac{i}{2}\}_t t^{j/2}$ &$\{\frac{i+j}{2}\}_t$ &$\{\frac{i+j}{2}\}_t$&$\{\frac{i+j}{2}\}_t$\\ 
& &&&&&\\
\hline 
\end{tabular}
\vspace{3mm}
\caption{The weight $\widetilde{Wt}(v; [x])$.}
\label{table2}
\end{table}

To find a relation with Table \ref{table1}, we need to replace Table \ref{table2} with another form, which is given by the following lemma.

\begin{lemma}
\label{modi}
For any $[x]\in \mathcal{T}/\mathord{\sim}$ we have 
$$\prod_{v\in V}\widetilde{Wt}(v;[x])=\prod_{v\in V}\widehat{Wt}(v;[x]),$$ where $\widehat{Wt}(v;[x])$ is defined in Table \ref{table3}.
\end{lemma}
\begin{proof}
The difference of $\widetilde{Wt}(v;[x])$ and $\widehat{Wt}(v;[x])$ appears around the solid curves of the state $\psi([x])$. Precisely
\begin{eqnarray*}
\frac{\widetilde{Wt}(v;[x])}{\widehat{Wt}(v;[x])}=\begin{cases}
t^{-i/2} & \psi([x]) \text{ has the form } \begin{tikzpicture}[baseline=-0.65ex, thick]
\draw  (0, -0.5) [->-] to (0,0);
\draw  (0,0) [->]  to  (0.66,0.5);
\draw  [dotted] (0,0) [->]  to  (-0.66,0.5);
\draw (-0.66, 0.75) node {$i$};
\draw (0.66, 0.75) node {$j$};
\draw (0, -0.75) node {$i+j$};
\end{tikzpicture} \text{ around $v$};\\
t^{i/2} & \psi([x]) \text{ has the form } \begin{tikzpicture}[baseline=-0.65ex, thick]
\draw  (0, 0) [->] to (0,0.5);
\draw   (0,0) [-<-]  to  (0.66,-0.5);
\draw [dotted]  (0,0) [-<-]  to  (-0.66,-0.5);
\draw (-0.66, -0.75) node {$i$};
\draw (0.66, -0.75) node {$j$};
\draw (0,0.75) node {$i+j$};
\end{tikzpicture} \text{ around $v$};\\
1 & \text{otherwise}.
\end{cases}
\end{eqnarray*} 
Therefore 
\begin{eqnarray*}
\prod_{v\in V} \frac{\widetilde{Wt}(v;[x])}{\widehat{Wt}(v;[x])}=\prod_{\text{$\gamma$: solid curve of $\psi([x])$}}\bigg(\prod_{v\in \gamma}\frac{\widetilde{Wt}(v;[x])}{\widehat{Wt}(v;[x])}\bigg)=\prod_{\text{$\gamma$: solid curve of $\psi([x])$}} 1=1.
\end{eqnarray*}
The first equality holds since the difference of $\widetilde{Wt}(v;[x])$ and $\widehat{Wt}(v;[x])$ only appears around the solid curves. Now we verify the second equality. Each solid curve $\gamma$ bounds a region $D\subset \mathbb{R}^2$ on its left-hand side when moving along $\gamma$. Let $\Gamma_{\mathrm{out}}$ (resp. $\Gamma_{\mathrm{in}}$) be the set of dotted edges in $D$ pointing out of (resp. into) $\gamma$. Then $\prod_{v\in \gamma}\frac{\widetilde{Wt}(v;[x])}{\widehat{Wt}(v;[x])}=t^{k/2}$, where
$$k=\sum_{e\in \Gamma_{\mathrm{out}}}c(e)-\sum_{e\in \Gamma_{\mathrm{in}}}c(e).$$ Note that after removing $\Gamma_{\mathrm{out}}\cup \Gamma_{\mathrm{in}}$, the graph $G$ becomes a disjoint union of two parts. Namely $$G\backslash (\Gamma_{\mathrm{out}}\cup \Gamma_{\mathrm{in}})=G_1 \cup G_2.$$ Then $\sum_{e\in \Gamma_{\mathrm{out}}}c(e)$ (resp. $\sum_{e\in \Gamma_{\mathrm{in}}}c(e)$) is the sum of colors of edges pointing from $G_1$ (resp. $G_2$) to $G_2$ (resp. $G_1$). Since $c$ is a balanced coloring, we see that $\sum_{e\in \Gamma_{\mathrm{out}}}c(e)=\sum_{e\in \Gamma_{\mathrm{in}}}c(e)$. Namely $k$ must be zero.
\end{proof}

\begin{comment}
\begin{figure}
\centering
\begin{tikzpicture}[baseline=-0.65ex, thick]
\draw (-1,0) ellipse (1cm and 2cm);
\draw (-2,0) [->-] to (-2, 0.1);
\draw (-2.5,-1.25) node {$\gamma$};
\draw [dotted] (0, 0) [->-] to (1,0);
\draw (1,-2) rectangle (3,2);
\draw (2,0) node {part of $G$};
\end{tikzpicture}
\end{figure}
\begin{figure}
\begin{tikzpicture}[baseline=-0.65ex, thick]
\draw (-1, 0.5) [->-]  to [out=270,in=180]  (0, 0) ;
\draw (0, 0) [->-] to (1, 0) ;
\draw (1, 0)[->-] to (2, 0);
\draw (2, 0) [->-] to (3, 0);
\draw (3,0) [->-] to [out=0,in=270]  (4, 0.5);
\draw (-1, 1) node {$\vdots$};
\draw (4, 1) node {$\vdots$};
\draw [dotted] (0, 0) [->] to (0,1);
\draw [dotted] (0, 0) [->] to (0,1);
\draw [dotted] (1, 0) to (1,-1);
\draw [dotted] (2, 1) [->-] to (2,0);
\draw [dotted] (3, 0) [->] to (3,1);
\draw (0,1.25) node {$i_1$};
\draw (2,1.25) node {$i_2$};
\draw (3,1.25) node {$i_3$};
\draw (2.5,-0.25) node {$c$};
\draw (-0.5, 1) to [out=270,in=180] (1.5, 0.3);
\draw (3.5, 1) to [out=270,in=0] (1.5, 0.3);
\draw [dotted] (3.5, 1) to [out=90,in=0] (1.5, 1.7);
\draw [dotted] (-0.5, 1) to [out=90,in=180] (1.5, 1.7);
\draw (3, 2.5) node {null-homologous in $H_1(S^3\backslash G)$};
\draw (3, 2) node{$\downarrow$};
\end{tikzpicture}
\end{figure}
\end{comment}

\begin{table}
\begin{tabular} {|c|c|c|c|c|c|c|} \hline
    State $\psi([x])$ & \begin{tikzpicture}[baseline=-0.65ex, thick]
\draw [dotted] (0, -0.5) [->-] to (0,0);
\draw  [dotted] (0,0) [->]  to  (0.66,0.5);
\draw  [dotted] (0,0) [->]  to  (-0.66,0.5);
\draw (-0.66, 0.75) node {$i$};
\draw (0.66, 0.75) node {$j$};
\draw (0, -0.75) node {$i+j$};
\end{tikzpicture}
     & \begin{tikzpicture}[baseline=-0.65ex, thick]
\draw  (0, -0.5) [->-] to (0,0);
\draw  (0,0) [->]  to  (0.66,0.5);
\draw  [dotted] (0,0) [->]  to  (-0.66,0.5);
\draw (-0.66, 0.75) node {$i$};
\draw (0.66, 0.75) node {$j$};
\draw (0, -0.75) node {$i+j$};
\end{tikzpicture}
      & \begin{tikzpicture}[baseline=-0.65ex, thick]
\draw  (0, -0.5) [->-] to (0,0);
\draw  [dotted] (0,0) [->]  to  (0.66,0.5);
\draw   (0,0) [->]  to  (-0.66,0.5);
\draw (-0.66, 0.75) node {$i$};
\draw (0.66, 0.75) node {$j$};
\draw (0, -0.75) node {$i+j$};
\end{tikzpicture}
& \begin{tikzpicture}[baseline=-0.65ex, thick]
\draw [dotted] (0, 0) [->] to (0,0.5);
\draw  [dotted] (0,0) [-<-]  to  (0.66,-0.5);
\draw [dotted]  (0,0) [-<-]  to  (-0.66,-0.5);
\draw (-0.66, -0.75) node {$i$};
\draw (0.66, -0.75) node {$j$};
\draw (0, 0.75) node {$i+j$};
\end{tikzpicture}
& \begin{tikzpicture}[baseline=-0.65ex, thick]
\draw  (0, 0) [->] to (0,0.5);
\draw   (0,0) [-<-]  to  (0.66,-0.5);
\draw [dotted]  (0,0) [-<-]  to  (-0.66,-0.5);
\draw (-0.66, -0.75) node {$i$};
\draw (0.66, -0.75) node {$j$};
\draw (0,0.75) node {$i+j$};
\end{tikzpicture}
& \begin{tikzpicture}[baseline=-0.65ex, thick]
\draw  (0, 0) [->] to (0,0.5);
\draw  [dotted] (0,0) [-<-]  to  (0.66,-0.5);
\draw   (0,0) [-<-]  to  (-0.66,-0.5);
\draw (-0.66, -0.75) node {$i$};
\draw (0.66, -0.75) node {$j$};
\draw (0, 0.75) node {$i+j$};
\end{tikzpicture}
       \\ \hline 
 \vspace{-2mm}& &&&&&\\
$\widehat{Wt}(v; [x])$ & $\{\frac{i+j}{2}\}_t$   & $\{\frac{j}{2}\}_t$ & $\{\frac{i}{2}\}_tt^{j/2}$ &$\{\frac{i+j}{2}\}_t$ &$\{\frac{i+j}{2}\}_tt^{-i/2}$&$\{\frac{i+j}{2}\}_t$\\ 
& &&&&&\\
\hline 
\end{tabular}
\vspace{3mm}
\caption{The weight $\widehat{Wt}(v; [x])$.}
\label{table3}
\end{table}

\vspace{3mm}

The Boltzmann weights in Table \ref{table1} are now related to Table \ref{table3} in the following way.

\begin{lemma}
\label{lC}
For each state $s\in \mathcal{S}$, we have $$\prod_{v\in V}\widehat{Wt}(v;s) \vert _{t=q^{-4}}=(\prod_{\text{$v$: even type}}\{2c(v)\}_q)\prod_{v\in V}Wt(v;s).$$
\end{lemma}
\begin{proof}
Compare the values of $Wt(v;s)$ and $\widehat{Wt}(v;s)$ in Tables \ref{table1} and \ref{table3}. The relation between them follows from that $(t^{i/2}-t^{-i/2}) \vert _{t=q^{-4}} =-(q^{2i}-q^{-2i})$ and that there are even number of vertices in $G$ (since the vertices of even type equals that of odd type). 
\end{proof}

Now we are ready to prove the following proposition, which is a weaker form of Theorem \ref{maintheorem}. It is weaker in the sense that we only consider plane graphs and it is an equality modulo $\pm \mathbb{Z}\langle t^{1/2} \rangle$.

\begin{prop}[A weaker form of Theorem \ref{maintheorem}]
\label{weaker}
$$\Delta_{(G, c)}(t) \vert _{t=q^{-4}}\overset{\cdot}{=} \frac{\prod_{\text{$v$: even type}}\{2c(v)\}_q}{(q^{2}-q^{-2})^{\vert V\vert -1}} \underline{\Delta}(G, c).$$
\end{prop}
\begin{proof}
In the following calculations, the first equality is Lemma \ref{lA}. The third equality is Lemma \ref{lB}. The fifth equality is Lemma \ref{modi}. The sixth equality is Lemma \ref{signeq}. The seventh equality is Lemma \ref{lC}. The last equality is Proposition \ref{vsum}.
\begin{eqnarray*} 
&&(t^{1/2}-t^{-1/2})^{\vert V\vert -1} \Delta_{(G, c)}(t)\\
&\overset{\cdot}{=}&\frac{1}{t^{j/2}-t^{-j/2}}\det (\phi_c(\frac{\partial \beta_v}{\partial \alpha_u}))_{u, v\in V}\\
&=&\frac{1}{t^{j/2}-t^{-j/2}} \sum_{\text{$\sigma$: bijection of $V$}} \mathrm{sign}(\sigma) \prod_{v\in V}\phi_c(\frac{\partial \beta_v}{\partial \alpha_{\sigma(v)}} ) \\
&\overset{\cdot}{=}& \frac{1}{t^{j/2}-t^{-j/2}} \sum_{\text{$\sigma$: bijection of $V$}} \mathrm{sign}(\sigma) \prod_{v\in V}(-1)^{\delta_{v, \sigma(v)}+1} \sum_{[x]\in \mathcal{T}/\mathord{\sim}, \sigma_x=\sigma} \widetilde{Wt}(v; [x])\\
&=&\frac{1}{t^{j/2}-t^{-j/2}}  \sum_{[x]\in \mathcal{T}/\mathord{\sim}}\mathrm{sign}(\sigma_x) \prod_{v\in V}(-1)^{\delta_{v, \sigma(v)}+1}\widetilde{Wt}(v;[x])\\
&=&\frac{1}{t^{j/2}-t^{-j/2}}  \sum_{[x]\in \mathcal{T}/\mathord{\sim}}\mathrm{sign}(\sigma_x) \prod_{v\in V}(-1)^{\delta_{v, \sigma(v)}+1}\widehat{Wt}(v;[x])\\
&=&\frac{1}{t^{j/2}-t^{-j/2}}   \sum_{s\in \mathcal{S}} \mathrm{sign}(s) \prod_{v\in V}\widehat{Wt}(v;s)\\
&\overset{\cdot}{=}&\frac{\prod_{\text{$v$: even type}}\{2c(v)\}_q}{q^{2j}-q^{-2j}} \sum_{s\in \mathcal{S}} \mathrm{sign}(s)\prod_{v\in V}Wt(v;s) \quad \quad (t=q^{-4})\\
 &\overset{\cdot}{=}&(\prod_{\text{$v$: even type}}\{2c(v)\}_q) \cdot \underline{\Delta}(G, c).
\end{eqnarray*}
\end{proof}

\bibliographystyle{siam}
\bibliography{bao}

\begin{thebibliography}{10}

\bibitem{MR1164114}
{\sc Y.~Akutsu, T.~Deguchi, and T.~Ohtsuki}, {\em Invariants of colored links},
  J. Knot Theory Ramifications, 1 (1992), pp.~161--184.

\bibitem{bao}
{\sc Y.~Bao}, {\em Floer homology and embedded bipartite graphs},
  arXiv:1401.6608v4,  (2018).

\bibitem{alex}
{\sc Y.~Bao and Z.~Wu}, {\em An {A}lexander polynomial for {MOY} graphs},
  arXiv:1708.09092v2,  (2017).

\bibitem{MR3073565}
{\sc F.~Costantino and J.~Murakami}, {\em On the {$\mathrm{SL}(2,\mathbb{C})$}
  quantum {$6j$}-symbols and their relation to the hyperbolic volume}, Quantum
  Topol., 4 (2013), pp.~303--351.

\bibitem{MR2805998}
{\sc S.~Friedl, A.~Juh{\'a}sz, and J.~Rasmussen}, {\em The decategorification
  of sutured {F}loer homology}, J. Topol., 4 (2011), pp.~431--478.

\bibitem{MR3470707}
{\sc J.~Grant}, {\em A generators and relations description of a representation
  category of {$U_q({gl}(1|1))$}}, Algebr. Geom. Topol., 16 (2016),
  pp.~509--539.

\bibitem{MR1133269}
{\sc L.~H. Kauffman and H.~Saleur}, {\em Free fermions and the
  {A}lexander-{C}onway polynomial}, Comm. Math. Phys., 141 (1991),
  pp.~293--327.

\bibitem{MR1417494}
{\sc A.~Kawauchi}, {\em A survey of knot theory}, Birkh\"auser Verlag, Basel,
  1996.
\newblock Translated and revised from the 1990 Japanese original by the author.

\bibitem{MR1659228}
{\sc H.~Murakami, T.~Ohtsuki, and S.~Yamada}, {\em Homfly polynomial via an
  invariant of colored plane graphs}, Enseign. Math. (2), 44 (1998),
  pp.~325--360.

\bibitem{MR1183500}
{\sc J.~Murakami}, {\em The multi-variable {A}lexander polynomial and a
  one-parameter family of representations of {$ U_q( sl(2,{\bf C}))$} at
  {$q^2=-1$}}, in Quantum groups ({L}eningrad, 1990), vol.~1510 of Lecture
  Notes in Math., Springer, Berlin, 1992, pp.~350--353.

\bibitem{MR1197048}
\leavevmode\vrule height 2pt depth -1.6pt width 23pt, {\em A state model for
  the multivariable {A}lexander polynomial}, Pacific J. Math., 157 (1993),
  pp.~109--135.

\bibitem{MR2529302}
{\sc P.~Ozsv{\'a}th, A.~Stipsicz, and Z.~Szab{\'o}}, {\em Floer homology and
  singular knots}, J. Topol., 2 (2009), pp.~380--404.

\bibitem{MR1223142}
{\sc N.~Reshetikhin}, {\em Quantum supergroups}, in Quantum field theory,
  statistical mechanics, quantum groups and topology ({C}oral {G}ables, {FL},
  1991), World Sci. Publ., River Edge, NJ, 1992, pp.~264--282.

\bibitem{MR1170953}
{\sc L.~Rozansky and H.~Saleur}, {\em Quantum field theory for the
  multi-variable {A}lexander-{C}onway polynomial}, Nuclear Phys. B, 376 (1992),
  pp.~461--509.

\bibitem{MR3319619}
{\sc A.~Sartori}, {\em The {A}lexander polynomial as quantum invariant of
  links}, Ark. Mat., 53 (2015), pp.~177--202.

\bibitem{MR2255851}
{\sc O.~Y. Viro}, {\em Quantum relatives of the {A}lexander polynomial},
  Algebra i Analiz, 18 (2006), pp.~63--157.

\end{thebibliography}

\end{document}